\documentclass[11pt]{article}
\usepackage[a4paper,margin=25mm]{geometry}
\usepackage[T1]{fontenc}
\usepackage{lmodern,microtype}
\usepackage{amsmath,amssymb,amsthm,mathtools,bm}
\usepackage{booktabs,array,enumitem}
\usepackage{xcolor,graphicx}
\usepackage{pgfplots}
\pgfplotsset{compat=1.18}
\usepackage[numbers,square,sort&compress]{natbib}
\usepackage{xurl}
\usepackage[colorlinks=true,linkcolor=blue!45!black,citecolor=blue!45!black,urlcolor=blue!45!black]{hyperref}
\usepackage{aliascnt}
\usepackage{cleveref}
\hypersetup{
 pdftitle={Rate-Optimal Spectral Likelihood-Ratio Clustering in Stochastic Block Models},
 pdfauthor={Zhixin Zhou},
 pdfsubject={Constant oracle-risk bounds for spectral likelihood-ratio clustering},
 pdfkeywords={stochastic block model, spectral likelihood-ratio clustering, community detection, oracle risk}
}
\setlist{itemsep=3pt,topsep=5pt,leftmargin=2em}
\allowdisplaybreaks[2]
\newtheorem{theorem}{Theorem}[section]
\newaliascnt{lemma}{theorem}
\newtheorem{lemma}[lemma]{Lemma}
\aliascntresetthe{lemma}
\newaliascnt{proposition}{theorem}
\newtheorem{proposition}[proposition]{Proposition}
\aliascntresetthe{proposition}
\newaliascnt{corollary}{theorem}
\newtheorem{corollary}[corollary]{Corollary}
\aliascntresetthe{corollary}
\theoremstyle{definition}
\newaliascnt{definition}{theorem}

\aliascntresetthe{definition}
\newaliascnt{assumption}{theorem}

\aliascntresetthe{assumption}
\theoremstyle{remark}
\newaliascnt{remark}{theorem}

\aliascntresetthe{remark}

\AddToHook{env/theorem/begin}{\crefalias{section}{theorem}}
\AddToHook{env/lemma/begin}{\crefalias{section}{lemma}}
\AddToHook{env/proposition/begin}{\crefalias{section}{proposition}}
\AddToHook{env/corollary/begin}{\crefalias{section}{corollary}}
\AddToHook{env/definition/begin}{\crefalias{section}{definition}}
\AddToHook{env/assumption/begin}{\crefalias{section}{assumption}}
\AddToHook{env/remark/begin}{\crefalias{section}{remark}}

\newcommand{\E}{\mathbb E}
\newcommand{\PP}{\mathbb P}
\newcommand{\one}{\mathbf 1}
\newcommand{\diag}{\operatorname{diag}}

\newcommand{\op}{\mathrm{op}}
\newcommand{\Ber}{\operatorname{Bern}}

\newcommand{\col}{\operatorname{col}}

\newcommand{\argmax}{\operatorname*{arg\,max}}

\newcommand{\norm}[1]{\left\lVert#1\right\rVert}
\newcommand{\abs}[1]{\left\lvert#1\right\rvert}

\newcommand{\ind}[1]{\mathbf1\{#1\}}
\newcommand{\Mis}{\operatorname{Mis}}
\title{Rate-Optimal Spectral Likelihood-Ratio\\Clustering in Stochastic Block Models}
\author{Zhixin Zhou}
\date{October 30, 2026}
\begin{document}
\maketitle
\begin{abstract}
Spectral clustering commonly assigns communities by clustering a
low-dimensional representation, whereas likelihood refinement can
improve an initial assignment using the observed edges. We ask whether
sharp community-recovery accuracy can be retained when the refinement
no longer accesses individual edges. We develop spectral likelihood-ratio
clustering for stochastic block models using one regularized adjacency
eigensystem and the original vertex degrees. The method estimates connection
probabilities and updates each vertex's label by a likelihood-ratio
classifier on reconstructed rows. The refinement accepts any sufficiently
accurate initialization, including one computed from the same graph.
Its expected misclassification proportion is bounded by a constant times
the error of a raw-edge oracle classifier, plus the probability that
initialization misses the required accuracy. Repeated $k$-means++ on
the retained spectral coordinates supplies one complete procedure with
a constant-factor oracle guarantee. The oracle knows the model parameters
and the other labels.
The guarantee holds for a fixed number of communities, comparable
positive block probabilities and community sizes, and a nondegenerate
normalized block matrix. It allows arbitrarily slowly diverging expected
degrees under the stated upper sparsity condition. We explain the
likelihood-ratio construction first, then the regularization and exact-degree
correction required by the general result. The proof combines sharp
local testing, truncated spectral-row estimates and leave-one-out
comparisons.
\end{abstract}
\section{Introduction}\label{sec:related-work}

The stochastic block model has a long history as a statistical model
for community structure in networks. \citet{holland1983} introduced the stochastic blockmodel
framework in 1983, in which connection probabilities are determined by the
communities of the two endpoints. Subsequent work developed likelihood
and Bayesian estimation of latent block structure~\citep{snijders1997,nowicki2001},
statistical formulations of community recovery~\citep{bickelchen2009},
and asymptotic theory for likelihood and variational
estimators~\citep{bickel2013}. A central question throughout this
literature is how accurately the unknown communities can be recovered
by a computationally feasible procedure.

\paragraph{Spectral clustering.}
A standard approach first represents each vertex by a small number of
spectral coordinates and then clusters these coordinates, typically by
$k$-means. \citet{rohe2011} studied Laplacian
spectral clustering, while \citet{sussman2012} analyzed adjacency spectral embedding.
\citet{lei2015} established consistency guarantees for
adjacency spectral clustering under sparse block models. These results
provide a statistical foundation for the familiar combination of a
spectral representation and a geometric clustering step.

Sparse graphs require additional care because unusually large or small
degrees can affect the spectral representation. Regularized spectral
procedures were studied by \citet{qin2013} and \citet{joseph2016}; \citet{llv2017}
established concentration results for regularized random graphs.
\citet{zhouamini2019} analyzed data-driven
regularization, spectral truncation and clustering in general bipartite
models, including expected degrees growing arbitrarily slowly.
\citet{chin2015} developed further spectral recovery
methods for sparse block models. Thus a logarithmic lower bound on
expected degree is not an intrinsic requirement of every spectral
approach.

\paragraph{Likelihood refinement and optimal error rates.}
Consistency alone does not determine whether the misclassification
rate is close to the best possible rate. A common way to improve an
initial clustering is to estimate its connection probabilities and
reclassify each vertex using its connections to the estimated
communities. Comparing two candidate labels is then a local
likelihood-ratio calculation, possibly with a penalty or a
pseudo-likelihood approximation. \citet{amini2013} developed scalable pseudo-likelihood methods
based on incident block counts. \citet{zhangzhou2016}
characterized minimax misclassification rates through local testing,
and \citet{gao2017} showed how penalized
local likelihood refinement can turn a weakly consistent initializer
into an estimator attaining optimal misclassification rates.

The distinction between a leading error exponent and an actual
constant-factor risk comparison matters here. \citet{zhouamini2020} combined spectral initialization with two
pseudo-likelihood classifications to obtain an oracle-optimal
biclustering rate up to constants. \citet{zhouli}
established sharp testing bounds that retain the prefactor and used
likelihood updates to obtain rate-optimal SBM recovery. These methods
show that accurate initialization need not itself be statistically
optimal: likelihood refinement can recover information that a purely
geometric assignment does not use. In these constructions the local
refinement uses observed edges or observed block counts, in addition
to the initial clustering.

\paragraph{The question addressed here.}
This development suggests a natural question: can a spectral
representation itself support the sharp recovery accuracy associated
with likelihood refinement, without returning to individual adjacency
entries? By a \emph{spectral representation} we mean the retained
leading eigenvectors together with their associated eigenvalues.
The question is not whether spectral initialization can be followed by
an edge-based likelihood update, but whether the update can instead
be computed from the retained representation.

There is already important progress on the accuracy of spectral
methods themselves. \citet{afwz} developed
entrywise eigenvector analysis with sharp community-recovery
consequences. \citet{zhang} obtained matching leading-exponent
bounds for spectral clustering through truncated spectral-row
analysis. Probabilistic classification of spectral coordinates has
also been studied, for example by \citet{suwan2016}. Our target is a constant-factor comparison
with a raw-edge likelihood oracle, rather than only a leading-exponent
comparison or a Gaussian model for the embedding coordinates.

We give an affirmative answer to this oracle-risk question for a
\emph{degree-augmented spectral representation}: one eigensystem of a
regularized adjacency matrix together with the original vertex degrees.
We call the resulting procedure \emph{spectral likelihood-ratio clustering}.
It reconstructs incident block counts, estimates block probabilities,
and updates the partition through vertexwise likelihood-ratio classification.
None of these later operations accesses individual adjacency entries.
The main theorem treats refinement separately from initialization. It
accepts any initial partition with error of order the reciprocal of
expected degree, even when the partition is computed from the same graph.
The expected error after refinement is at most a fixed constant times
that of the raw-edge oracle, plus the probability that the initializer
fails to reach the required accuracy. The oracle knows the true connection
probabilities and the other vertices' community labels. Repeated
$k$-means++ gives one end-to-end procedure for which the additional
failure term is of oracle order. The result allows
arbitrarily slowly diverging expected degrees on the positive,
well-conditioned model class specified in Section~\ref{sec:statements}.
The upper sparsity restriction and all other assumptions are stated in
Section~\ref{sec:main-result}; no logarithmic degree lower bound is
imposed.

The spectral reconstruction estimates how a vertex's connections are
distributed across communities. The observed degree supplies their
exact total, which is not in general preserved by spectral truncation.
This separation motivates the degree correction in the complete
procedure. We first explain a spectrum-only construction to make the
likelihood idea transparent, and then give the degree-augmented
procedure covered by the theorem. The result is therefore a statement
about this retained representation and its likelihood refinement, not a
claim that $k$-means alone or the preliminary uncorrected version has the
same sharp risk. The link from approximate geometric clustering to an
initial misclassification bound follows the organization of
\citet[Corollary 1 and Remark 8]{zhouamini2019}; the randomized
seeding guarantee comes from \citet{arthur2007}. We separately account
for its failure probability rather than treating a constant-probability
approximation as an oracle-scale guarantee. Section~\ref{sec:oracle-scale} relates the oracle
error to the sharp testing exponent used by \citet{zhouli},
including the square-root prefactor, and distinguishes this comparison
from a classwise minimax statement.

\paragraph{Other recovery objectives and model extensions.}
Misclassification risk is one of several recovery criteria for SBMs.
Detection transitions were investigated by \citet{decelle2011}, with sharp two-community results by
\citet{mns2018} and
\citet{massoulie2014}; nonbacktracking methods provide
another spectral approach in very sparse networks~\citep{krzakala2013}.
Exact-recovery thresholds and efficient procedures were studied by
\citet{abh2016} and \citet{abbesandon2015}, including semidefinite approaches by
\citet{hajek2016} and \citet{aminilevina2018}. \citet{abbe2018} surveys these
recovery notions. Our average-risk guarantee concerns a different
comparison and yields exact recovery when the total oracle error
vanishes; it does not address detection at bounded expected degree.

Recovery of latent structure provides a complementary set of objectives.
Random dot product graph models assign latent vectors to vertices and
use their inner products as connection probabilities~\citep{young2007}.
\citet{sussman2014} established consistency of adjacency-spectral latent
position estimates and studied supervised vertex classification from
those estimates. \citet{rubindelanchy2022} developed a generalized
random dot product graph model that permits indefinite connectivity
structure, together with consistency and asymptotic normality results
for spectral embeddings under their assumptions. This distinction
matters for SBMs: the ordinary dot product representation requires a
positive-semidefinite block matrix, whereas the generalized model also
allows negative eigenvalues. A related notion of optimality concerns
the connectivity matrix and its singular subspaces rather than the
vertex labels. In sparse bipartite Bernoulli models,
\citet{zhou2021subspace} established minimax-optimal rates for estimating
low-rank connectivity matrices and their singular subspaces under
matrix and projection losses. These results place spectral methods in
a broader theory of latent-structure estimation, but optimal estimation
of a matrix or subspace does not by itself establish a sharp
misclassification-risk bound. Our refinement result in
Theorem~\ref{restore:sbm-theorem} addresses this further classification
question for the SBM: whether the retained spectral representation,
augmented by degrees, supports a constant-factor comparison with the
raw-edge oracle. It does not assert the same guarantee for general
random dot product graphs.

Degree heterogeneity is a separate modeling issue. \citet{karrer2011} introduced a degree-corrected blockmodel;
related consistency, spectral and minimax results include
\citet{zhao2012}, \citet{jin2015}, and
\citet{gmzz2018}. In the ordinary SBM studied
here, preserving observed degrees corrects a reconstruction; it does
not introduce vertex-specific degree parameters into the sampling model.

Section~\ref{sec:statements} introduces the stochastic block model,
notation and oracle. Section~\ref{sec:algorithm} describes spectral likelihood-ratio clustering
and its intuition, starting with the simple spectral construction.
Section~\ref{sec:main-result} states the theorem, describes its sharp
error scale and gives an outline of the proof. Section~\ref{sec:experiments}
compares the empirical risk with the oracle as the graph size increases
and reports implementation checks for the growing-size experiment. The seven
appendices contain all supporting arguments.

\section{Stochastic Block Model}\label{sec:statements}

Let $z\in[K]^n$ be the community labels and $P=P^\top\in [0,1]^{K\times K}$ the block
probability matrix, where $K\ge2$ is fixed. Write
$n_k=\sum_i\ind{z_i=k}$. The adjacency matrix $A$ is symmetric, has
zero diagonal, and has independent upper-triangular entries with
\begin{equation}\label{eq:mean-model}
 A_{ij}\sim\Ber(P_{z_i z_j}),\qquad i<j.
\end{equation}
Put $p=\max_{k,\ell\in[K]}P_{k\ell}$. For fixed $\beta\ge1$,
$\omega>1$, and a fixed positive signal constant, assume
\begin{equation}\label{eq:structural}
 n_k\ge\frac{n}{\beta K},\qquad
 0<\frac{p}{\omega}\le P_{k\ell}\le p,\qquad
 \sigma_K(P)\ge c p.
\end{equation}
Here $\sigma_K(P)$ is the smallest singular value of $P$. The matrix $P$ may vary with $n$. For nonnegative
quantities, $f\lesssim g$ means $f\le Cg$, $f\gtrsim g$ means
$g\lesssim f$, and $f\asymp g$ means both. Unless stated otherwise,
the implicit constants may depend only on fixed $K,\beta,\omega$,
the fixed positive lower bound on $\sigma_K(P)/p$, the uniform upper
bound on $np^2$, and the fixed algorithm constants $b_0,c_0$ introduced
in Section~\ref{restore:algorithm}. A refinement bound may additionally
depend on the fixed constant in its stated initialization-accuracy
condition. An end-to-end randomized procedure may also have a fixed
restart multiplier. These constants never depend on $n$, on the parameters
within this class, or on the number of likelihood refits.
Local lemmas state any further dependence on fixed thresholds.
We retain explicit $c,C$ when their choice or position in a bound
matters; their values may change between occurrences. Vertex indices
are $i,j,j'$ and community indices are $k,\ell$, with $r$ used
when a third community is needed.
For any label vector $\widehat z$, use the same size convention
$n_k(\widehat z)=\sum_i\ind{\widehat z_i=k}$, so $n_k=n_k(z)$.

The misclassification proportion is
\[
 \Mis(\widehat z,z)=\frac1n\min_{\pi\in S_K}
     \sum_i\ind{\pi(\widehat z_i)\ne z_i},
\]
where $S_K$ is the set of permutations of $[K]$.
Following the classification-operator organization of \citet{zhouli},
write the row-classification formula as
\begin{equation}\label{eq:sbm-oracle}
 [\mathcal L(A,P,\widehat z)]_i
 =\argmax_{k\in[K]}\sum_{j\ne i}
   \{A_{ij}\log P_{k\widehat z_j}+(1-A_{ij})\log(1-P_{k\widehat z_j})\}.
\end{equation}
Enumerate the communities in their fixed tie order, so a tie is won
by the smallest index. For the raw graph the oracle classifier outputs
$[\mathcal L(A,P,z)]_i$: it uses the true neighbor labels but not
$z_i$. An exact global size constraint is not used to infer that label.

The oracle experiment depends on the true community, not on which
vertex in that community is tested. For a representative target in
community $k$, let $b_{r}$ be its raw count of edges to community
$r$. Under $\PP_k$, these counts are independent and
\[
 b_{r}\sim\operatorname{Bin}
       \bigl(n_{r}-\ind{r=k},P_{kr}\bigr),
 \qquad r=1,\ldots,K.
\]
There are $n_k-1$ trials in the target's own community and $n_{r}$
in every other community. The self-loop is deterministically zero
under every candidate label and contributes no likelihood term; it
is not an additional Bernoulli nonedge. For $\ell\ne k$, the exact
log-likelihood ratio in favor of $k$ against $\ell$ is
\[
T_{k\ell}=\sum_{r=1}^K\Bigl[b_{r}\log\frac{P_{kr}}{P_{\ell r}}+(n_{r}-\ind{r=k}-b_{r})\log\frac{1-P_{kr}}{1-P_{\ell r}}\Bigr].
\]
For each fixed $k$, all the contrasts $T_{k\ell}$ use the same count
vector; they are not assumed independent. Different vertices in that
community give different realizations of this same experiment, so no
vertex index is needed in its distribution or risk. The average oracle
error is exactly
\begin{equation}\label{eq:oracle-risk}
 R_{\mathrm{oracle}}
 =\sum_{k=1}^K\frac{n_k}{n}\,
   \PP_k\!\left\{
     \min_{\ell<k}T_{k\ell}\le0
     \ \text{or}\ 
     \min_{\ell>k}T_{k\ell}<0
   \right\}.
\end{equation}
A minimum over an empty set is $+\infty$. The two inequalities in
\eqref{eq:oracle-risk} implement the fixed tie rule. This is the same
average error as $n^{-1}\sum_i\PP\{[\mathcal L(A,P,z)]_i\ne z_i\}$,
not a new oracle. Section~\ref{sec:oracle-scale}, specifically
\eqref{eq:oracle-scale-comparison}, relates it to the Chernoff
information and shows $R_{\mathrm{oracle}}\asymp e^{-I}$ with the
sharp prefactor included in $I$.

\section{Spectral likelihood-ratio clustering}\label{sec:algorithm}

For fixed model parameters and neighbor labels, the log-likelihood ratio
between two candidate communities depends on the incident block counts
and community sizes. Replacing raw counts by their spectral reconstructions
gives a vertexwise likelihood-ratio classifier. Fitting the block
probabilities and applying this classifier to all vertices updates the
estimated partition without reading individual edges again. We first
give the spectrum-only construction; the general procedure then adds
regularization and exact preservation of the observed row totals.

\subsection{The spectrum-only construction}\label{sec:simple}

Retain the $K$ eigenpairs of $A$ with largest absolute eigenvalues and
form
\[
 AU=U\Lambda,\qquad U^\top U=I_K,\qquad \widehat A=U\Lambda U^\top.
\]
Here $\widehat A$ is the rank-$K$ spectral reconstruction, before
any degree correction. Eigenvalues are selected by magnitude, so
informative negative ones are retained. The likelihood and parameter
fits below always exclude the diagonal. This section uses only
$(U,\Lambda)$ after the spectral calculation, not the observed
degree vector.

The vertexwise classifier in \eqref{eq:sbm-oracle} is evaluated at
its first argument: this section substitutes $\widehat A$ for $A$,
and Section~\ref{restore:algorithm} substitutes $\widehat A^{\mathrm{c}}$.
These reconstructions can be signed or asymmetric. Their classification
criterion is a working likelihood, not an assertion that their entries
are independent Bernoulli observations. The term \emph{spectral
likelihood-ratio} refers to comparisons of this criterion, not to an
exact joint likelihood for the retained eigenpairs.

\subsubsection{Block counts and likelihood-ratio classification}

Using the same first-argument convention, define the block counts by
\begin{equation}\label{eq:simple-counts}
 b_{i\ell}(A,\widehat z)=\sum_{j\ne i:\widehat z_j=\ell}A_{ij}.
\end{equation}
In particular, the spectral counts are computable from the factors:
\begin{equation}\label{eq:simple-factor-counts}
 b_{i\ell}(\widehat A,\widehat z)
 =U_{i\cdot}\Lambda
       \sum_{j:\widehat z_j=\ell}U_{j\cdot}^\top
       -\widehat A_{ii}\ind{\widehat z_i=\ell}.
\end{equation}
The expression uses only the vertex's retained spectral coordinates
and the sums of those coordinates within the estimated communities.
Thus changing the current labels changes the block counts through
these sums, without another eigendecomposition or edge lookup.
The reconstructed counts need not be integers or nonnegative. Grouping the criterion
in \eqref{eq:sbm-oracle} by the current labels gives the candidate-$k$
score
\begin{equation}\label{eq:simple-score}
 \sum_{\ell=1}^K\left\{
 b_{i\ell}(\widehat A,\widehat z)\log P_{k\ell}
 +[n_\ell(\widehat z)-\ind{\widehat z_i=\ell}-b_{i\ell}(\widehat A,\widehat z)]
       \log(1-P_{k\ell})\right\}.
\end{equation}
Thus the raw-row oracle, the simple spectral update, and the general
update below use the same classification operator with different
matrix and parameter inputs. The full nonedge term is retained.

The probability fit is just a block average. For this step write
$b_{i\ell}=b_{i\ell}(\widehat A,\widehat z)$ and compute
\begin{equation}\label{eq:block-average-operator}
 \widehat P_{k\ell}\leftarrow
 \begin{cases}
 \dfrac{\sum_{i:\widehat z_i=k} b_{i\ell}}
       {n_k(\widehat z)\,n_\ell(\widehat z)},
       & k\ne\ell,\\[8pt]
 \dfrac{\sum_{i:\widehat z_i=k} b_{ik}}
       {n_k(\widehat z)\,[n_k(\widehat z)-1]},
       & k=\ell.
 \end{cases}
\end{equation}
Both cases average the reconstructed counts over community $k$ and
then divide by the number of possible neighbors. The second excludes
self-connections. Symmetry of $\widehat A$ makes this fit symmetric. Finally, replace estimates below $n^{-2}$
by $n^{-2}$ and those above $1/2$ by $1/2$; estimates between these
bounds are unchanged. These restrictions apply only to $\widehat P$,
not to the reconstructed counts. A zero denominator invokes the
fixed fallback below.

\subsubsection{Initialization and iteration}\label{sec:initialization}

Initialization and likelihood refinement have different tasks. The
initializer need only provide a partition with error of order $(np)^{-1}$;
it is not itself required to attain the oracle risk. The refinement
result in Section~\ref{sec:main-result} is not tied to one clustering
routine. It allows the initial labels to depend on the same graph as
the subsequent reconstruction and parameter fit.

One standard choice is $k$-means++ on row-normalized spectral coordinates.
Set $x_i=U_{i\cdot}/\norm{U_{i\cdot}}_2$ for a nonzero row, and assign a
fixed unit vector to a zero row. Choose the first center uniformly from
these $n$ points. Choose each of the remaining $K-1$ centers with
probability proportional to the squared distance to its nearest already
chosen center, and assign each point to its nearest center. This is the
original squared-distance sampling rule of \citet{arthur2007}, rather
than a different greedy seeding variant. If all sampling weights vanish,
fill the remaining center positions using unused vertex indices in a
fixed order.

A prescribed finite number of Lloyd steps may follow: replace each
nonempty cluster center by its mean, keep the previous center for an
empty cluster, and reassign to nearest centers. Each step must leave
the within-cluster sum of squared distances nonincreasing. Independent
runs are compared using this same sum, retaining the smallest one;
true labels, true probabilities, and oracle errors are never used to
select a run. Distance and run-selection ties have fixed rules. An
empty fitted group invokes the fixed fallback. The proof shows that
this cannot occur on the successful low-cost event.

The choice of initialization precision and its probability are distinct.
On the regularized spectral event, the ideal normalized centers are
separated and their total squared approximation error is $O(1/p)$.
A constant-factor $k$-means fit therefore yields $O((np)^{-1})$
misclassification. A single $k$-means++ run supplies such a fit only
with a fixed success probability under the general approximation
bound. Repetition is needed to obtain the sharper failure guarantee
used by the complete procedure; Section~\ref{restore:algorithm} gives
a degree-based restart rule, and Appendix~\ref{restore:fit} proves its
conditional-on-the-graph guarantee. This role of the geometric step is
consistent with \citet[Corollary 1 and Remark 8]{zhouamini2019}.
For the unregularized illustrative procedure here, the number of runs
is a prescribed positive integer; no sparse oracle-risk guarantee is
asserted for this preliminary construction.

Starting from the chosen initial labels $\widehat z^{(0)}$, set
$\widehat z=\widehat z^{(0)}$. Compute $\widehat P$ by
\eqref{eq:block-average-operator}, apply the stated probability bounds,
and then update
\begin{equation}\label{eq:simple-fit}
 \widehat z\leftarrow\mathcal L(\widehat A,\widehat P,\widehat z).
\end{equation}
Repeat the probability fit and label update for a specified finite
positive number of rounds. Label updates are simultaneous, with the
right-hand side evaluated at the pre-update labels. Empty block-pair
sets or other undefined operations use the fixed fallback.
The probability bounds depend only on $n$, so no degree or true
probability parameter is required. For $n\ge K+1$ the interval is nonempty.

The block mean in \eqref{eq:block-average-operator}, after the stated
probability bounds are applied, maximizes its Bernoulli working
criterion, including when a reconstructed block total is signed.
This fact does not establish convergence or monotonicity of the
simultaneous likelihood label updates. Nonincrease of the geometric
Lloyd objective is a separate statement and is not used to claim
monotonicity of the likelihood iterations.
The spectrum-only procedure illustrates the likelihood-ratio
construction; the guarantees below concern the corrected reconstruction
of the complete procedure next.

\subsection{The general degree-plus-spectrum algorithm}
\label{restore:algorithm}

The general clustering procedure retains raw degrees in addition to
spectral information. It regularizes high-degree vertices before computing
$\widehat A$, then corrects the row totals. The block-average fit
and vertexwise likelihood-ratio classifier are given explicitly below.
Fix sufficiently large $b_0>0$ and sufficiently small
$c_0>0$, depending only on the structural class. The refinement steps
and their tie rules are deterministic once the initial partition is
provided. Initialization may be randomized; undefined operations use
the fixed fallback.

\paragraph{Regularized adjacency spectrum.}
Let
\[
 D_i=\sum_jA_{ij},\qquad \bar D=\frac1n\sum_iD_i,
 \qquad \tau=b_0\bar D.
\]
The all-zero graph has the fixed output. Otherwise form the diagonal
matrix
\[
 W_i=\begin{cases}
 1,&D_i\le\tau,\\
 \sqrt{\tau/D_i},&D_i>\tau,
 \end{cases}
 \qquad W=\diag(W_i),\qquad \widetilde A=WAW.
\]
Retain the raw degrees $D_i$ and the $K$ eigenpairs of $\widetilde A$
with largest absolute eigenvalues:
\[
 \widetilde AU=U\Lambda,\qquad U^\top U=I_K.
\]
The decoder uses only these retained quantities; $\bar D,\tau,W$ are
recovered from the degrees. Edges between vertices below the threshold
are unchanged by the cap.

\paragraph{Original-scale spectral reconstruction and degree correction.}
The eigenpairs here belong to $\widetilde A=WAW$, so
$U\Lambda U^\top$ is a reconstruction on the regularized scale.
On the original adjacency scale, set
\[
 \widehat A=W^{-1}U\Lambda U^\top W^{-1}.
\]
For $W=I$ this reduces to $\widehat A=U\Lambda U^\top$, exactly
as in the simple version. Write $\widehat A^{\mathrm{c}}$ for its
degree-preserving, zero-diagonal correction:
\begin{equation}\label{restore:reconstruction}
 \widehat A^{\mathrm{c}}_{ii}=0,\qquad
 \widehat A^{\mathrm{c}}_{ij}=\widehat A_{ij}
      +\frac{D_i-\sum_{j'\ne i}\widehat A_{ij'}}{n-1},\quad j\ne i.
\end{equation}
Then $\sum_j\widehat A^{\mathrm{c}}_{ij}=D_i$ exactly. The entries are not clipped;
$\widehat A^{\mathrm{c}}$ can be signed and asymmetric. An isolated vertex has a
zero reconstructed row and does not halt classification of other
vertices. Everything in \eqref{restore:reconstruction} is computed
from the retained degrees and eigenpairs.

For the same labels $\widehat z$, the single count definition
\eqref{eq:simple-counts} gives
\begin{equation}\label{eq:count-bridge}
 \begin{aligned}
 b_{i\ell}(\widehat A^{\mathrm{c}},\widehat z)-b_{i\ell}(\widehat A,\widehat z)
 &=\frac{n_\ell(\widehat z)-\ind{\widehat z_i=\ell}}{n-1}
     \left(D_i-\sum_{j\ne i}\widehat A_{ij}\right),\\
 \sum_{\ell=1}^K b_{i\ell}(\widehat A^{\mathrm{c}},\widehat z)&=D_i.
 \end{aligned}
\end{equation}
At fixed $\widehat z$ and $\widehat P$, the resulting change in the
candidate-$k$ versus candidate-$\ell$ log-likelihood contrast is
\begin{equation}\label{eq:score-bridge}
 \left(D_i-\sum_{j\ne i}\widehat A_{ij}\right)
 \frac1{n-1}\sum_{j\ne i}
 \log\frac{\widehat P_{k\widehat z_j}(1-\widehat P_{\ell \widehat z_j})}
               {\widehat P_{\ell \widehat z_j}(1-\widehat P_{k\widehat z_j})}.
\end{equation}
This is an identity for one update at fixed inputs, not for the final
outputs of two differently fitted procedures. An inactive cap does
not make the degree correction vanish.

\paragraph{Initialization.}
Supply initial labels $\widehat z^{(0)}$. The general refinement theorem
allows any method satisfying its stated accuracy requirement; the method
may use the same full graph. To make the entire procedure depend only on
the retained representation, compute these labels from the retained
eigenpairs and degrees. One such choice is the row-normalized
$k$-means++ procedure of Section~\ref{sec:initialization}, applied to this
regularized $U$. Run it independently $\max\{1,\lceil C\bar D\rceil\}$
times and retain the smallest within-cluster squared-distance objective,
where the fixed multiplier $C$ is sufficiently large for the structural
class. Optional Lloyd steps are finite and nonincreasing as specified
there. Neither the unknown $p$ nor true labels enter this restart rule.
A smaller practical restart budget remains covered by the general
bound with its initialization failure term, but is not automatically
covered by the end-to-end corollary.

\paragraph{Block fitting and classification.}
Set $\widehat z=\widehat z^{(0)}$. Given current labels, recompute
$b_{i\ell}=b_{i\ell}(\widehat A^{\mathrm{c}},\widehat z)$
and form the block averages in \eqref{eq:block-average-operator}.
Before applying any probability bounds, average the two directions:
\begin{equation}\label{eq:block-fit}
 \widehat P\leftarrow\frac{\widehat P+\widehat P^\top}{2}.
\end{equation}
This is needed because $\widehat A^{\mathrm{c}}$ can be asymmetric.
Then replace estimates below $c_0\bar D/n$ by $c_0\bar D/n$ and
those above $1/2$ by $1/2$. The simple version's lower bound $n^{-2}$
is not used here. Only $\widehat P$ is symmetrized and restricted;
$\widehat A^{\mathrm{c}}$, its degree sums, and its classification
rows remain unchanged.
Update all labels simultaneously by
\begin{equation}\label{eq:fitted-row-score}
 \widehat z\leftarrow\mathcal L(\widehat A^{\mathrm{c}},\widehat P,\widehat z).
\end{equation}
These are precisely the symmetric Bernoulli block fit and the full
edge-and-nonedge row likelihood. There is no target degree parameter
to profile in this ordinary SBM algorithm. Empty groups, empty pair
sets, or undefined operations use the fixed fallback.

Perform at least one completed fit and label update. Any finite number
of these hard refits is permitted, including finite data-dependent
stopping. The proof shows that the required operations are well-defined
on the aggregate event when the initial partition has the stipulated
accuracy. The repeated $k$-means++ implementation uses no true labels,
true parameters, or leave-one-out eigensystems. Once initial labels
have been supplied, the refinement reads only those labels and the
retained degrees and eigenpairs.

\section{Main theorem and proof outline}\label{sec:main-result}

The first result concerns the refinement steps of
Section~\ref{restore:algorithm}, including their original-degree
correction. It separates the accuracy of an input partition from the
risk achieved by refinement. The input partition may be computed from
the same graph used for all subsequent steps. The spectrum-only
construction of Section~\ref{sec:simple} is not assigned this guarantee.

\begin{theorem}[Refinement from a same-graph initialization]
\label{restore:sbm-theorem}
Under \eqref{eq:mean-model}--\eqref{eq:structural}, suppose
\[
 np\longrightarrow\infty,\qquad np^2\le C_1<\infty.
\]
Fix $0<C_0<\infty$. Let $\widehat z^{(0)}$ be any measurable initialization
computed from $A$, possibly with auxiliary randomization. For sufficiently
large fixed $b_0$ and sufficiently small fixed $c_0>0$, start the
block-fit and spectral likelihood-ratio updates of
Section~\ref{restore:algorithm} at $\widehat z^{(0)}$. After any finite
positive number of these updates, their output $\widehat z$ satisfies
\begin{equation}\label{eq:target}
 \E\Mis(\widehat z,z)
 \lesssim R_{\mathrm{oracle}}
 +\PP\!\left\{\Mis(\widehat z^{(0)},z)>\frac{C_0}{np}\right\}.
\end{equation}
The same inequality holds for any almost surely finite, positive,
data-dependent number of updates. Probability and expectation include
any initialization randomness. The implicit constant may depend on
$C_0$ and the fixed quantities in Section~\ref{sec:statements}, but not
on $n$, the initialization rule, or the number of refits. No independence
between the initial labels and any entries of $A$ is assumed.
\end{theorem}

The failure probability in \eqref{eq:target} cannot in general be
replaced by an expected initial misclassification bound. An
$O((np)^{-1})$ expected initial error, or success with probability
$1-o(1)$, does not by itself pay for a much smaller oracle risk.
If the displayed initialization failure probability is
$O(R_{\mathrm{oracle}})$, the refined risk is of oracle order.
One likelihood update already suffices; additional specified updates
do not accumulate failure probabilities.

\begin{corollary}[Repeated $k$-means++ and spectral likelihood-ratio refinement]
\label{cor:kmpp-oracle}
Under the same model conditions, use the degree-based repeated
$k$-means++ initialization in Section~\ref{restore:algorithm} with a
sufficiently large fixed restart multiplier. Its input labels and
all subsequent updates depend only on the retained degrees and
spectral representation. After any finite positive number of the
specified likelihood updates, including almost surely finite adaptive
stopping, the output satisfies
\begin{equation}\label{eq:kmpp-oracle-risk}
 \E\Mis(\widehat z,z)\lesssim R_{\mathrm{oracle}}.
\end{equation}
\end{corollary}

The initializer supplies a coarse $O((np)^{-1})$ accuracy; likelihood
refinement supplies the sharp risk comparison. Both results permit
$np=\log\log n$ and arbitrarily slower divergence. They do not impose
$np\ge c\log n$, a planted partition, or equal expected degrees.
They retain the probability-comparability and signal conditions in
\eqref{eq:structural}, as well as the upper bound on $np^2$.
Increasing the number of seeding runs is a computational requirement,
not a lower density assumption.
The risk ratio is bounded, not asserted to tend to one. The corollary
implies exact recovery whenever $nR_{\mathrm{oracle}}\to0$.
For a general initializer, the corresponding sufficient condition is
that $n$ times the entire right side of \eqref{eq:target} tends to zero.

\subsection{The sharp oracle scale and its interpretation}\label{sec:oracle-scale}

For the community-level contrast $T_{k\ell}$, define the adjusted
Chernoff information using the full community sizes:
\[
 I_{k\ell}=-\min_{0\le\alpha\le1}\sum_{r=1}^K n_{r}
 \log\!\left[P_{kr}^{1-\alpha}P_{\ell r}^\alpha
 +(1-P_{kr})^{1-\alpha}(1-P_{\ell r})^\alpha\right]
 +\frac12\log(np).
\]
Put $I=\min_{k\ne\ell}I_{k\ell}$. Following the sharp testing scale
in \citet{zhouli}, the added logarithm includes the square-root
prefactor in the exponent. We use $n_{r}$ rather than
$n_{r}-\ind{r=k}$ only in this information expression;
$T_{k\ell}$ and the oracle risk in \eqref{eq:oracle-risk} remain exact.
Adding the single missing Bernoulli factor increases the exact adjusted
Chernoff information by an amount between zero and
$-\log(1-p)=O(p)$, as shown in Appendix~\ref{sec:sharp}.
Thus $e^{-I_{k\ell}}$ is a risk scale, not an exact probability, and
this simplification changes that scale by at most a fixed factor.
The full-size expression satisfies $I_{k\ell}=I_{\ell k}$ by
replacing $\alpha$ with $1-\alpha$; the exact directional risks need
not be equal.

Under the fixed structural assumptions and $np\to\infty$,
Appendix~\ref{sec:sharp} proves, uniformly over $k\ne\ell$ and
bounded $x$,
\[
 \PP_k(T_{k\ell}<x)\asymp
 \PP_k(T_{k\ell}\le x)\asymp e^{-I_{k\ell}}.
\]
A strict defeat by any competitor forces an oracle error, and an
oracle error entails a non-strict defeat by some competitor. For
fixed $K$, \eqref{eq:oracle-risk} therefore gives
\begin{equation}\label{eq:oracle-scale-comparison}
 \begin{aligned}
 R_{\mathrm{oracle}}
 &\asymp\sum_{k=1}^K\frac{n_k}{n}
              \sum_{\ell\ne k}\PP_k(T_{k\ell}\le0)\\
 &\asymp\sum_{k=1}^K\frac{n_k}{n}
              \sum_{\ell\ne k}e^{-I_{k\ell}}
 \asymp e^{-I}.
 \end{aligned}
\end{equation}
For the final comparison, a pair attaining $I$ has weight at least
$1/(\beta K)$, whereas the total weight of all pairs is $K-1$.
No equality of the pairwise exponents, and no independence of the
competing contrasts, is needed. If the initialization failure term in
\eqref{eq:target} is $O(R_{\mathrm{oracle}})$, as in
\cref{cor:kmpp-oracle}, the conclusion is
\[
 \E\Mis(\widehat z,z)\lesssim R_{\mathrm{oracle}}\asymp e^{-I}.
\]
Here $I$ is a sharp error exponent, not the unmodified Chernoff
information or an exact negative log risk. The proof always tilts the
actual self-excluded row law. Its exact exponent can be replaced by
$I_{k\ell}$ in constant-factor bounds, but not in exact
change-of-measure identities without the one-coordinate correction.

A minimax lower bound is a statement about the worst-case risk over
a specified parameter class, as in
\citep{zhangzhou2016,zhouli,zhouamini2020}. Where such a lower bound
matches this sharp scale on the same class, a uniform oracle upper
bound yields minimax-rate optimality. That reasoning does not give
a pointwise lower bound for this estimator at every parameter value.
Accordingly, the theorem retains the upper inequality for the
algorithmic risk; the two-sided comparison above is for the oracle.

\subsection{Outline of the proof}\label{sec:proof-strategy}

\paragraph{Initialization is a separate input.}
The refinement theorem requires $\Mis(\widehat z^{(0)},z)\le C_0/(np)$
only on the event where initialization succeeds, and charges its
complement explicitly. Aggregate spectral estimates provide a
normalized-coordinate squared error of order $1/p$. Any fixed-factor
$k$-means approximation then provides the required initial accuracy.
For $k$-means++, independent restarts reduce its conditional failure
probability to $2^{-m}$ after $m$ runs. The degree-based restart rule
makes this term small enough relative to the oracle; the argument is
proved in Appendix~\ref{restore:fit} and completed in
Appendix~\ref{sec:probability}.

\paragraph{Uniform fitted-score control before conditioning.}
The initial labels and fitted probabilities depend on $A$, so they
cannot first be held fixed in an independent-edge calculation.
Instead, on one aggregate graph event, the following estimates hold
simultaneously for every sufficiently accurate partition:
\[
 \max_{k,\ell}|\widehat P_{k\ell}-P_{k\ell}|
 \lesssim n^{-1}+p\Mis(\widehat z,z),
\]
and the fitted contrast differs from the true-parameter contrast on
the same reconstructed row by $O(1+np\Mis(\widehat z,z))$.
The key deterministic estimate controls
$\sum_{j\in E}|\widehat A^{\mathrm c}_{ij}|\lesssim 1+p|E|$
for every vertex set $E$, including an error set selected using the
whole graph. There is no union bound over partitions and no
independence assumption on the error set. Successful initialization
and the invariant error bound make the score remainder uniformly
bounded over all refits.

\paragraph{Leave one out only after removing the fitted quantities.}
For a tested vertex $i$, recompute the degrees, cap threshold, weights,
and eigenspace from $B=A_{-i,-i}$. The local comparison in
Appendix~\ref{restore:loo} converts the true-parameter reconstructed
contrast into the actual raw-row oracle contrast plus a perturbation
whose coefficients are measurable from $B$. The preceding deterministic
comparison first bounds the error event on initialization success and
the full-graph local events. These events are then dropped; only the
training event determined by $B$ is retained before conditioning on
$B$. This order, made explicit in Appendix~\ref{sec:probability},
restores incident-edge independence without asserting independence
conditional on initial labels or initialization success. No stability
of the initializer under deleting a vertex is required, and no
leave-one-out algorithm is run in practice.

\paragraph{Sharp risk and event costs.}
At the oracle tilt, projection identities bound the mean and quadratic
mass of the coefficient error. Truncated spectral-row estimates charge
large coefficients only when the tested row encounters them. The sharp
comparison preserves the prefactor included in $I$.
The polynomial and exponential aggregate bounds in
Appendices~\ref{sec:sparse-extension} and~\ref{sec:high-range} cover
complementary degree ranges; this proof split is not a density
restriction. Averaging the local testing bounds pays all graph and
local exceptions at oracle order, and adds the initialization failure
probability just once. No union over iteration counts is used.

All seven proof appendices remain part of this ordinary-SBM article.
No degree-corrected-SBM extension or companion note is needed for its
main theorem.

\section{Experiments}\label{sec:experiments}

We examine the asymptotic comparison in the theorem by increasing the
number of vertices while keeping the block-probability profiles and
community proportions fixed. The quantity of interest is the average
misclassification proportion divided by $R_{\mathrm{oracle}}$, rather
than the number of errors in a single graph. We use the $k$-means++
initializer of Section~\ref{sec:initialization} followed by the
spectral likelihood-ratio refinement of Section~\ref{restore:algorithm}.

\subsection{Model sequences and implementation}

The graph sizes are
\[
 n=2000,\ 4000,\ 8000,\ 16000,\ 32000,\ 64000.
\]
For each size and each of the following three sequences, we use
100 independent Bernoulli graphs, giving 1800 graphs in total.
The balanced two-community model has
\[
 P=p\begin{pmatrix}1&0.25\\0.25&1\end{pmatrix},
 \qquad np=4\log n.
\]
The three-community model has proportions $(0.40,0.35,0.25)$ and
\[
 P=p\begin{pmatrix}
 1&0.10&0.15\\
 0.10&0.90&0.12\\
 0.15&0.12&0.80
 \end{pmatrix},\qquad np=6\log n.
\]
Its community sizes and expected degrees are unequal. The third sequence
uses the balanced two-community matrix with $np=6\sqrt{\log n}$.
All profiles and proportions remain fixed, and all three sequences
satisfy $np\to\infty$ and $np^2\to0$. The third degree scale grows
more slowly than $\log n$ asymptotically, although it still exceeds
$\log n$ on the finite grid above.

We fix $b_0=3$ and $c_0=0.01$ at every size. The encoder retains the
$K$ eigenpairs of largest absolute eigenvalue, with relative eigensolver
tolerance $10^{-10}$. On the normalized rows of $U$, initialization uses
the original $D^2$ rule: a uniform first center, followed by sampling
proportional to squared distance from the existing centers. There are
$\max\{1,\lceil2\bar D\rceil\}$ independent starts, each allowing at
most 50 nonincreasing Lloyd steps. An empty cluster keeps its previous
center. We select the smallest within-cluster squared error, with the
first start winning a tie. Numerically zero rows, with norm at most
$10^{-14}$, receive the fixed unit direction. The numerical multiplier
two is held fixed; it is not asserted to be the sufficient constant
for the whole parameter class in \cref{cor:kmpp-oracle}.

After initialization we allow at most ten complete parameter-fit and
simultaneous LR updates, stopping earlier when the labels are unchanged.
The decoder receives only the initial partition, original degrees,
eigenvectors and eigenvalues. It uses neither adjacency entries nor
true labels or probabilities. All graph outcomes enter the averages;
restart selection and stopping do not use the observed error or oracle.

\paragraph{Oracle risk and Monte Carlo uncertainty.}
For each model and size, we compute $R_{\mathrm{oracle}}$ by summing
the independent block-binomial probabilities over the error event in
\eqref{eq:oracle-risk}. The calculation includes the full nonedge
term, the $n_k-1$ own-community trials and the fixed tie rule. The
discarded binomial tail mass is below $10^{-13}$. Thus the denominator
is the specified Bernoulli oracle risk, not $e^{-I}$ or an observed
error fraction from a single graph.

The numerator is the mean misclassification proportion across the 100
graphs. We divide that mean and its Student-$t$ 95\% Monte Carlo
interval by the fixed oracle denominator. The independent replicates
are graphs, not individual vertices; the intervals are pointwise rather
than simultaneous. No multiplicative factor is fitted to align the
algorithm and oracle curves.

\subsection{Risk ratios as the graph size increases}

Table~\ref{tab:experiment-risks} reports the two risks and their ratio,
and Figure~\ref{fig:experiment-ratios} displays the ratios over the
32-fold increase in graph size. The ranges are 2.74--2.96, 2.15--2.48, and 3.16--3.29 in the three sequences, respectively.
The ratios remain on a constant scale and do not exhibit growth with
$n$ on these three sequences. This is numerical evidence consistent
with the asymptotic comparison; the ratios are not asserted to tend
to one.

For the balanced logarithmic-degree sequence, increasing $n$ from
2000 to 64000 reduces the algorithmic risk from $5.600\times10^{-3}$ to
$8.202\times10^{-4}$, while the oracle risk decreases from $1.893\times10^{-3}$ to
$2.992\times10^{-4}$. The corresponding ratio changes from 2.96 to
2.74. For the unequal three-community sequence, the two
risks decrease from $3.050\times10^{-3}$ and $1.228\times10^{-3}$ to
$2.642\times10^{-4}$ and $1.231\times10^{-4}$, respectively. Thus the observed
constant-scale ratios compare two decreasing risks.

The three-community setting also illustrates the effect of refinement.
At $n=64000$, mean misclassification is 0.1125\% after
initialization, 0.02648\% after one LR update, and
0.02642\% at the final output. Most of the improvement occurs
in the first update. In the balanced two-community settings,
initialization and final risks are much closer, so the evidence does
not imply a large refinement gain in every model.

\begin{table}[htbp]
\centering\small
\caption{Risk estimates for $k$-means++ followed by spectral LR, using all 100 independent graphs in each model-size cell. The denominator is the numerically evaluated Bernoulli oracle risk. Intervals quantify graph-level Monte Carlo uncertainty.}\label{tab:experiment-risks}
\begin{tabular}{rcccc}
\toprule
$n$ & $R_{\mathrm{oracle}}$ & Mean $\Mis$ & Risk ratio & 95\% interval\\
\midrule
\multicolumn{5}{l}{\textit{Balanced $K=2$, $np=4\log n$}}\\[2pt]
2,000 & $1.893\times10^{-3}$ & $5.600\times10^{-3}$ & 2.96 & [2.78, 3.14]\\
4,000 & $1.315\times10^{-3}$ & $3.805\times10^{-3}$ & 2.89 & [2.75, 3.04]\\
8,000 & $9.094\times10^{-4}$ & $2.654\times10^{-3}$ & 2.92 & [2.78, 3.05]\\
16,000 & $6.280\times10^{-4}$ & $1.806\times10^{-3}$ & 2.88 & [2.78, 2.97]\\
32,000 & $4.334\times10^{-4}$ & $1.230\times10^{-3}$ & 2.84 & [2.74, 2.94]\\
64,000 & $2.992\times10^{-4}$ & $8.202\times10^{-4}$ & 2.74 & [2.67, 2.82]\\
\midrule
\multicolumn{5}{l}{\textit{Unequal-size $K=3$, $np=6\log n$}}\\[2pt]
2,000 & $1.228\times10^{-3}$ & $3.050\times10^{-3}$ & 2.48 & [2.31, 2.66]\\
4,000 & $7.981\times10^{-4}$ & $1.735\times10^{-3}$ & 2.17 & [2.01, 2.34]\\
8,000 & $4.913\times10^{-4}$ & $1.076\times10^{-3}$ & 2.19 & [2.04, 2.34]\\
16,000 & $3.119\times10^{-4}$ & $6.706\times10^{-4}$ & 2.15 & [2.02, 2.28]\\
32,000 & $1.977\times10^{-4}$ & $4.259\times10^{-4}$ & 2.15 & [2.04, 2.27]\\
64,000 & $1.231\times10^{-4}$ & $2.642\times10^{-4}$ & 2.15 & [2.05, 2.24]\\
\midrule
\multicolumn{5}{l}{\textit{Balanced $K=2$, $np=6\sqrt{\log n}$}}\\[2pt]
2,000 & $1.349\times10^{-2}$ & $4.419\times10^{-2}$ & 3.28 & [3.19, 3.36]\\
4,000 & $1.219\times10^{-2}$ & $4.017\times10^{-2}$ & 3.29 & [3.24, 3.35]\\
8,000 & $1.105\times10^{-2}$ & $3.608\times10^{-2}$ & 3.26 & [3.22, 3.31]\\
16,000 & $1.005\times10^{-2}$ & $3.246\times10^{-2}$ & 3.23 & [3.20, 3.26]\\
32,000 & $9.160\times10^{-3}$ & $2.922\times10^{-2}$ & 3.19 & [3.17, 3.21]\\
64,000 & $8.377\times10^{-3}$ & $2.650\times10^{-2}$ & 3.16 & [3.15, 3.18]\\
\bottomrule\end{tabular}\end{table}

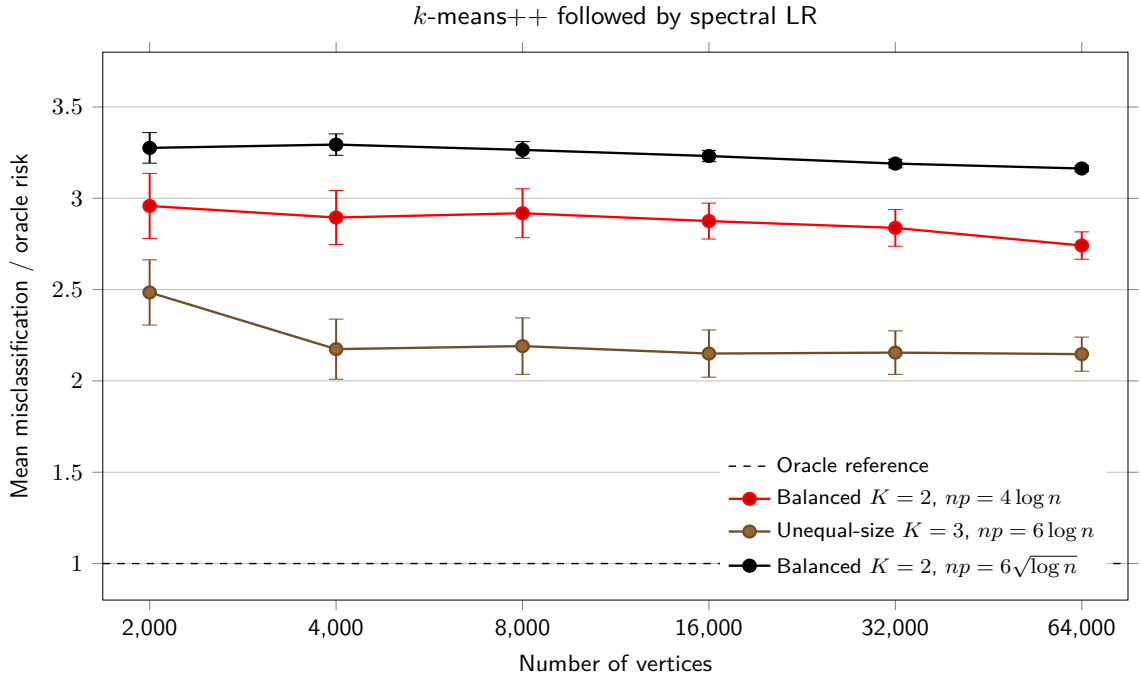
\begin{figure}[htbp]
\centering
\resizebox{.94\linewidth}{.563774290658\linewidth}{%
\begin{tikzpicture}
\begin{axis}[
 width=465pt,height=278pt,
 xmode=log,log basis x=2,
 xmin=1680,xmax=76000,ymin=.8,ymax=3.8,
 xtick={2000,4000,8000,16000,32000,64000},
 xticklabels={{2,000},{4,000},{8,000},{16,000},{32,000},{64,000}},
 ytick={1,1.5,2,2.5,3,3.5},
 minor x tick num=0,minor y tick num=0,
 tick align=outside,
 tick label style={font=\sffamily\small},
 xlabel={Number of vertices},
 ylabel={Mean misclassification / oracle risk},
 title={$k$-means++ followed by spectral LR},
 label style={font=\sffamily\small},
 title style={font=\sffamily\normalsize},
 ymajorgrids=true,
 legend style={at={(.98,.025)},anchor=south east,
   font=\sffamily\footnotesize,draw=none,row sep=2pt,cells={anchor=west}},
 legend cell align=left,
]
\addplot[dashed,line width=.6pt,domain=1680:76000,samples=2] {1};
\addlegendentry{Oracle reference}
\addplot+[mark=*,mark size=2.5pt,line width=1pt,
 error bars/.cd,y dir=both,y explicit,
 error bar style={line width=.8pt},error mark options={rotate=90,mark size=3pt}]
 table[x=n,y=ratio,y error plus=plus,y error minus=minus,row sep=\\] {
n ratio plus minus\\
2000 2.957923085372194 0.1784478650932844 0.1784478650932848\\
4000 2.894581475155128 0.1479211016584601 0.1479211016584614\\
8000 2.918209296118861 0.1339457042129668 0.1339457042129659\\
16000 2.875202965159536 0.09827298749209001 0.09827298749209046\\
32000 2.837761443019752 0.1003055144982277 0.1003055144982277\\
64000 2.741194233484076 0.07499093295017767 0.07499093295017856\\
};
\addlegendentry{Balanced $K=2$, $np=4\log n$}
\addplot+[mark=*,mark size=2.5pt,line width=1pt,
 error bars/.cd,y dir=both,y explicit,
 error bar style={line width=.8pt},error mark options={rotate=90,mark size=3pt}]
 table[x=n,y=ratio,y error plus=plus,y error minus=minus,row sep=\\] {
n ratio plus minus\\
2000 2.484440284840352 0.1785023143604096 0.1785023143604096\\
4000 2.173983769518058 0.1645263450723138 0.1645263450723133\\
8000 2.190569905211302 0.1543767047386577 0.1543767047386577\\
16000 2.149809650986965 0.1293279486671044 0.129327948667104\\
32000 2.15494834055247 0.1189494165875415 0.118949416587542\\
64000 2.146539554428367 0.09342985831017669 0.09342985831017625\\
};
\addlegendentry{Unequal-size $K=3$, $np=6\log n$}
\addplot+[mark=*,mark size=2.5pt,line width=1pt,
 error bars/.cd,y dir=both,y explicit,
 error bar style={line width=.8pt},error mark options={rotate=90,mark size=3pt}]
 table[x=n,y=ratio,y error plus=plus,y error minus=minus,row sep=\\] {
n ratio plus minus\\
2000 3.275818041372713 0.08353761598019949 0.08353761598019904\\
4000 3.294032278027191 0.05867995167071305 0.0586799516707126\\
8000 3.264978931302178 0.04530824582692627 0.04530824582692583\\
16000 3.231349786566815 0.03093595151659612 0.03093595151659612\\
32000 3.18991599172248 0.02344962237046966 0.02344962237046921\\
64000 3.163018039933997 0.0169313527560635 0.01693135275606394\\
};
\addlegendentry{Balanced $K=2$, $np=6\sqrt{\log n}$}
\end{axis}
\end{tikzpicture}%
}
\caption{Risk ratios as $n$ increases from 2000 to 64000. Each point includes 100 independent graphs, with pointwise 95\% Monte Carlo intervals. The oracle reference is one, and the curves have not been rescaled.}\label{fig:experiment-ratios}
\end{figure}

\subsection{Numerical checks and reproducibility}

All 1800 runs reach unchanged labels within 7 LR rounds,
without invoking the fallback or activating the probability bounds.
The degree-based restart rule uses 21--56 starts
per graph. Across the primary grid there are 92 isolated
vertices, all handled by the stated rule, and only 1 vertex is
subject to the high-degree cap. These samples therefore primarily
test the degree-preserving decoder rather than frequent cap activation.

The initializer is checked against a separate literal implementation
on 80 small point sets, including zero rows, duplicates and ties;
720 additional checks verify selection by minimum squared error.
The LR calculations are checked against direct dense formulas in
60 cases, including active caps, negative retained eigenvalues and
isolated vertices. Counts, probability fits, degree sums and scores
agree within $7.2\times10^{-15}$. Exhaustive small-row enumeration
also checks the block-binomial oracle calculation. The largest relative
eigen-equation residual on the primary grid is $4.640\times10^{-11}$.

The 1500 primary records at $n=2000$ through $32000$ are retained
from the existing $k$-means++ simulations; 300 new graphs extend the
same protocol to $n=64000$. Algorithm constants, seed rules, model
profiles and stopping rules are unchanged. The accompanying archive
contains the source code, protocols, graph and initialization seeds,
all graph-level outcomes, per-start objectives and oracle calculations.
The earlier smaller-size and weak-signal sensitivity runs remain
archived separately; the present comparison focuses on the six
increasing sizes above.

\section*{Acknowledgements}

The central idea of spectral likelihood-ratio clustering and the formulation
of the research problem originated with the author. The detailed proof
arguments in Appendices~\ref{sec:sharp}--\ref{sec:probability} and the
experimental code underlying Section~\ref{sec:experiments} were developed
primarily using OpenAI's GPT-6.0 and GPT-6.1 through iterative interactions
directed by the author. The models also assisted with algorithmic details,
debugging, and manuscript drafting and revision. Their contribution was
substantive and was not limited to language editing.

The author provided mathematical guidance and critical feedback, directed
the revisions, and made the final decisions about the work. The author
assumes full responsibility for the paper's content, including the
mathematical claims, code, numerical results, and any remaining errors.

\clearpage

\appendix
\section{A sharp Bernoulli likelihood transfer}\label{sec:sharp}

We give the probability step explicitly to distinguish a risk factor from
an exponential-rate factor. Sharp Chernoff comparisons in community
detection are also developed by \citet{zhouli}.
The argument below allows a perturbation measurable from independent
training data.

\begin{lemma}[Bounded tilted error preserves the sharp prefactor]
\label{lem:sharp}
Let $X_j$ be independent Bernoulli variables with probabilities $p_j$.
Let $T$ be an affine function of $(X_j)_j$, with linear coefficients
$w_j$, and put $\psi(\alpha)=\log\E e^{-\alpha T}$.
Suppose $\abs{w_j}\le C$, the minimizer $\alpha^*$ satisfies
$c\le \alpha^*\le1-c$ for a fixed $c>0$, and under the probability
measure $\mathbb{Q}$ proportional to $e^{-\alpha^*T}$,
\begin{equation}\label{eq:sharp-premises}
 \E_{\mathbb{Q}} T=0,\quad c{np}\le\operatorname{Var}_{\mathbb{Q}}(T)\le C{np},
 \quad\sum_jr_j\le C{np},
\end{equation}
where $r_j=\mathbb{Q}(X_j=1)$. Define
\[
 I=-\psi(\alpha^*)+\frac12\log(np).
\]
Within this lemma, $I$ is the exact transform-based exponent of $T$.
The added logarithm incorporates the square-root prefactor; hence
$I-\tfrac12\log(np)$ is the negative log transform at the optimizing
tilt. Its comparison with the full-size $I_{k\ell}$ of
Section~\ref{sec:oracle-scale} is made in the application below.
For every fixed $L<\infty$ and all sufficiently large ${np}$, uniformly over
$|x|\le L$,
\begin{equation}\label{eq:sharp-oracle}
 e^{-I}\lesssim\PP(T<x)
 \le\PP(T\le x)\lesssim e^{-I}.
\end{equation}
The implicit constants may depend on $L$ and the fixed constants in
the oracle hypotheses, but not on $n$ or on $x$ within this interval.

Let $\mathcal W$ be training data independent of $(X_j)_j$ and let
\[
 \widehat T=T+\varepsilon_0+\sum_j\varepsilon_jX_j
\]
have coefficients measurable with respect to $\mathcal W$. On a measurable
training event suppose
\begin{equation}\label{eq:tilt-error}
 \max_j|\varepsilon_j|\le R,\qquad
 \left|\varepsilon_0+\sum_jr_j\varepsilon_j\right|\le B_1,\qquad
 \sum_jr_j\varepsilon_j^2\le B_2.
\end{equation}
For fixed $R,B_2$ and every real $x$,
\begin{equation}\label{eq:sharp-transfer}
 \PP(\widehat T\le x\mid\mathcal W)
 \lesssim\exp\{\alpha^*x+B_1+e^R B_2/2\}\,e^{-I}.
\end{equation}
The implicit factor may depend on the fixed oracle bounds and on
$R,B_2$, but is uniform in $x,B_1,n$ and over the training event.
In particular $B_1$ can
be a nonnegative training-measurable random bound, provided its indicated
exponential moment is integrated when removing the conditioning.
\end{lemma}

\begin{proof}
Under $\mathbb{Q}$, the coordinates remain independent Bernoulli variables,
\[
 r_j=\frac{p_je^{-\alpha^*w_j}}{1-p_j+p_je^{-\alpha^*w_j}}.
\]
The sum of their centered weighted third absolute moments is at most
$C\sum_jr_j=O({np})$. Berry--Esseen and
\eqref{eq:sharp-premises} therefore show that every interval of length one
has $\mathbb{Q}$ probability at most $C/\sqrt{np}$, uniformly in
its location. The exact change of measure is
\[
 \frac{\mathrm{d}\PP}{\mathrm{d}\mathbb{Q}}
 =\exp\{\psi(\alpha^*)+\alpha^*T\}
 =\sqrt{np}\,e^{-I+\alpha^*T}.
\]
Partitioning $(-\infty,x]$ into unit intervals cancels the
$\sqrt{np}$ factor against the interval-probability bound and gives
the upper bound in \eqref{eq:sharp-oracle}.

For the lower bound choose a fixed interval $(x-L_0,x)$, with $L_0$
sufficiently large. Its Gaussian probability under a centered normal
variable of variance in $[c{np},C{np}]$ is at least $cL_0/\sqrt {np}$ for large
${np}$, uniformly over bounded $x$. Choose $L_0$ so this exceeds twice the
Berry--Esseen error. The $\mathbb{Q}$ probability is then at least $c'/\sqrt{np}$.
On this interval $e^{\alpha^*T}$ has a fixed positive lower bound,
so the same change-of-measure identity gives $c_Le^{-I}$.
Using strict endpoints also handles arbitrary lattice atoms and tie rules.

For the perturbation, independence of training and the row allows us to
regard $\varepsilon_0,\varepsilon$ as fixed. The elementary Bernoulli exponential-moment bound
gives
\begin{align*}
 \log\E e^{-\alpha^*\widehat T}
 &= -I+\tfrac12\log(np)-\alpha^*\varepsilon_0+\log\E_{\mathbb{Q}}e^{-\alpha^*\sum_j\varepsilon_jX_j}\\
 &\le -I+\tfrac12\log(np)-\alpha^*\left(\varepsilon_0+\sum_jr_j\varepsilon_j\right)
       +\frac{(\alpha^*)^2e^{\alpha^*R}}2\sum_jr_j\varepsilon_j^2\\
 &\le -I+\tfrac12\log(np)+B_1+e^R B_2/2.
\end{align*}
Now tilt the original law by $e^{-\alpha^*\widehat T}$.
Write $\widehat r_j=\widehat{\mathbb Q}(X_j=1)$.
Under this new measure $\widehat{\mathbb{Q}}$ the coordinates remain independent and
their odds are those under $\mathbb{Q}$ multiplied by $e^{-\alpha^*\varepsilon_j}$.
Consequently their probabilities and variances are within fixed factors
of those under $\mathbb{Q}$. Furthermore,
\[
 \left(\sum_jr_j(1-r_j)(w_j+\varepsilon_j)^2\right)^{1/2}
 \ge \sqrt{c{np}}-\sqrt{B_2}.
\]
For large ${np}$, the variance of $\widehat T$ under $\widehat{\mathbb{Q}}$ is therefore
between positive constant multiples of ${np}$. Its centered third absolute
moments sum to $O({np})$ because $w_j+\varepsilon_j$ are uniformly bounded and
$\sum_j\widehat r_j=O({np})$. Berry--Esseen bounds every unit interval under
$\widehat{\mathbb{Q}}$ by $C/\sqrt {np}$, irrespective of its mean. Multiplication by the displayed exponential-moment bound
cancels the $1/\sqrt{np}$ interval factor. The same interval summation
therefore proves \eqref{eq:sharp-transfer}.
\end{proof}

\subsection{Application to the actual Bernoulli SBM oracle}

Fix a tested vertex with true label $k$ and realize the count vector
from Section~\ref{sec:statements} using its actual incident edges.
For this fixed target, $T_{k\ell}$ is therefore exactly its raw-row
contrast, not an independent copy. Its distribution depends only on
$k,\ell$, the community sizes and $P$. The vertex is still indexed by
$i$ wherever its deleted graph or projection is used.
For the projection arguments the linear coefficient vector is
\begin{equation}\label{eq:sbm-weights}
 (w_{k\ell})_j
   =\log\frac{P_{kz_j}(1-P_{\ell z_j})}
                    {P_{\ell z_j}(1-P_{kz_j})},\qquad j\in[n].
\end{equation}
Thus, for the chosen target,
\[
 T_{k\ell}=\sum_{j\ne i}A_{ij}(w_{k\ell})_j
       +\sum_{j\ne i}\log\frac{1-P_{kz_j}}{1-P_{\ell z_j}}.
\]
The nonedge term is retained, and $w_{k\ell,-i}$ means that coordinate
$i$ is deleted. The independent binomial counts give the exact transform
\begin{equation}\label{eq:sbm-information}
 \log\E_k e^{-\alpha T_{k\ell}}
 =\sum_{r=1}^K(n_{r}-\ind{r=k})\,
 \log\!\left[P_{kr}^{1-\alpha}P_{\ell r}^\alpha
 +(1-P_{kr})^{1-\alpha}(1-P_{\ell r})^\alpha\right].
\end{equation}
The self-loop contributes nothing. To compare this exact transform
with the full-size expression in Section~\ref{sec:oracle-scale}, note
that the additional affinity factor obeys, for $0\le\alpha\le1$,
\[
 1-p\le P_{kk}^{1-\alpha}P_{\ell k}^{\alpha}
 +(1-P_{kk})^{1-\alpha}(1-P_{\ell k})^{\alpha}\le1.
\]
The lower bound follows from the nonedge term, and the upper bound
from H\"older's inequality. Taking negative logarithms and maximizing
shows that
\[
 0\le I_{k\ell}
       +\min_{0\le\alpha\le1}\log\E_k e^{-\alpha T_{k\ell}}
       -\tfrac12\log(np)\le-\log(1-p).
\]
Thus the exact transform-based exponent in \cref{lem:sharp} is
$I_{k\ell}+O(p)$, and its exponential lies between $e^{-I_{k\ell}}$
and $(1-p)^{-1}e^{-I_{k\ell}}$. Under $np^2\le C_1$, eventually
$p\le1/2$, so this additional multiplier is uniformly at most two.
Every optimizing tilt and change-of-measure identity in the lemma
still uses the exact transform \eqref{eq:sbm-information}, not a
fictitious observed self-loop. The optimizer of the full-size
expression need not coincide with that exact optimizer.
On the compact positive full-rank classes used here, probabilities are
comparable to $p$ and bounded away from one, the weights are bounded,
$\alpha^*$ stays away from the endpoints, and the tilted variance is comparable
to ${np}$. To verify the last assertions, the derivatives of the transform
at its endpoints are respectively minus and plus the two directional
Bernoulli KL divergences, both of order ${np}$. Its second derivative is
bounded by $C{np}$. The row separation implied by
$\sigma_K(P)\ge c p$ gives a lower bound $c{np}$ on that derivative
throughout $[0,1]$: tilted probabilities are comparable to $p$ and the
sum of squared log-odds differences is at least $cn$.
Thus \cref{lem:sharp} applies uniformly to every pair. Combining it
with the one-coordinate comparison above gives, for every fixed
bounded threshold $x$, the bounds
\[
 \PP_k(T_{k\ell}\le x)\asymp e^{-I_{k\ell}},
 \qquad -\log\PP_k(T_{k\ell}\le x)=I_{k\ell}+O(1).
\]
These are constant-factor comparisons, not exact risk identities.
The use of $\tfrac12\log(np)$ is justified by the tilted variance
being comparable to $np$ and the optimizer staying away from the
endpoints on the stated parameter class. It is not a universal
prefactor formula for vanishing signal separation.

\subsection{Fixed-\texorpdfstring{$K$}{K} comparison and label alignment}

A strict pairwise defeat of the true community forces an oracle
error; an error entails a non-strict defeat by some competitor.
Under the joint community-$k$ count law this gives
\[
 \begin{split}
 \max_{\ell\ne k}\PP_k(T_{k\ell}<0)
 &\le\PP_k\!\left\{
       \min_{\ell<k}T_{k\ell}\le0
       \ \text{or}\ \min_{\ell>k}T_{k\ell}<0\right\}\\
 &\le\sum_{\ell\ne k}\PP_k(T_{k\ell}\le0).
 \end{split}
\]
The sharp strict and non-strict tails, together with fixed $K$, imply
\begin{equation}\label{eq:oracle-pairs}
 R_{\mathrm{oracle}}
 \asymp\sum_{k=1}^K\frac{n_k}{n}
                    \sum_{\ell\ne k}e^{-I_{k\ell}}.
\end{equation}
The weighting $n_k/n$ is the original average over vertices, regrouped
by their true community. This is not a redefinition as a worst-pair
oracle. The argument requires no independence among competing tests.
Fitted partitions are aligned with the true groups by a single common
permutation on the good event. Uniform score comparisons under that
alignment imply
\[
 \E[\Mis(\widehat z,z)\ind{\mathcal G}]
 \le\frac1n\sum_i\PP(\text{aligned fitted error at }i,\mathcal G).
\]
The ensuing pairwise bounds apply to fixed true coefficient vectors; a
data-dependent matching does not become an independence assumption.

\subsection{Exceptional coefficients}

\begin{lemma}[Paying exceptional coefficients only when the row hits them]
\label{lem:sparse-truncated-transfer}
Assume the oracle hypotheses of \cref{lem:sharp}, and write $p_j$
for the un-tilted row probabilities. Suppose $p_j\le Cr_j$, as
holds here for bounded oracle likelihood coefficients. On an independent
training event $\mathcal E$, suppose
\[
 \left|\varepsilon_0+\sum_jr_j\varepsilon_j\right|\le B_1,\qquad \sum_jr_j\varepsilon_j^2\le B_2.
\]
Here $B_1,B_2$ are fixed constants. For any fixed $R>0$ and bounded
$x$, the exceptional coordinates can be charged directly:
\begin{equation}\label{eq:sparse-truncated-transfer}
 \begin{split}
 \PP(\widehat T\le x)
 \le{}& C_Re^{-I}+\PP(\mathcal E^c)\\
 &+\E\!\left[\ind{\mathcal E}
             \sum_{j:|\varepsilon_j|>R}p_j\right].
 \end{split}
\end{equation}
The last term is the expected probability bound for the tested row
hitting an exceptional coefficient. Since
\[
 \sum_{j:|\varepsilon_j|>R}p_j
 \le\frac{C}{R^2}
      \sum_jr_j\varepsilon_j^2\ind{|\varepsilon_j|>R},
\]
a corresponding bound for the truncated quadratic error also suffices.
Either expected term of order $e^{-I}$ is enough.
If $I\le Cnp$ for a fixed $C$, an exceptional-term bound
$C_se^{-snp}$ is sufficient for any fixed $s>C$.
\end{lemma}

\begin{proof}
For analysis only, set
$\varepsilon'_j=\varepsilon_j\ind{|\varepsilon_j|\le R}$.
Then
\[
 \left|\varepsilon_0+\sum_jr_j\varepsilon'_j\right|
 \le B_1+\frac1R\sum_{j:|\varepsilon_j|>R}r_j\varepsilon_j^2
 \le B_1+B_2/R,\qquad
 \sum_jr_j(\varepsilon'_j)^2\le B_2.
\]
The bounded-coefficient lemma bounds the truncated score's error by
$C_Re^{-I}$ on $\mathcal E$.
The two random scores agree whenever $X_j=0$ on every exceptional
coordinate. Conditional on training, the probability they differ is
at most $\sum_{j:|\varepsilon_j|>R}p_j$. Integrating this bound and
paying $\PP(\mathcal E^c)$ proves the result. The truncation is a
proof device, not an additional operation in the decoder.
\end{proof}

\section{Aggregate geometry for the capped adjacency matrix}
\label{sec:sparse-extension}

Let $Z_{ik}=\ind{z_i=k}$ be the community indicator matrix.
The completed matrix $ZPZ^\top$ includes its diagonal, whereas
$\E A=ZPZ^\top-\diag(P_{z_i z_i})$. Their difference has norm at
most $p$ and is retained in the bounds below. The completed row
sums $\sum_jP_{z_i z_j}$ are block constant and comparable to $np$.
In this appendix put $\tau=b_0\bar D$ and
$G=\diag\{\max(D_i,\tau)\}$, with diagonal entries $G_i$.
Then $W=\tau^{1/2}G^{-1/2}$ is the algorithm's high-degree cap,
with $b_0$ sufficiently large for the structural class. Thus
$\widetilde A=WAW$ leaves every edge between vertices of degree at most $\tau$
unchanged. For analysis only, write
$Q=\widetilde A/\tau=G^{-1/2}AG^{-1/2}$ and
$\norm v_G^2=v^\top Gv$. The scalar rescaling $Q$ has the same
eigenvectors as the encoded adjacency matrix $\widetilde A$.

\begin{proposition}[Aggregate capped-adjacency geometry]\label{prop:sparse-floor}
Under \eqref{eq:structural}, let $np\to\infty$ and $np^2\le C_1$.
For every fixed $r>0$, with probability at least $1-C_rn^{-r}$,
$c{np}\le\tau\le C{np}$, $\tau\ge\max_i\sum_jP_{z_i z_j}$,
$c{n^2p}\le\sum_iD_i$, and
\begin{align}
 \norm{W(A-{ZPZ^\top})W}_{\op}&\le C_r\sqrt {np},
 &\sum_iG_i&\le C_r{n^2p},\label{eq:floor-noise}\\
 \sum_i\frac{(D_i-\sum_jP_{z_i z_j})^2}{G_i}
 +\sum_i\frac{(G_i-\tau)^2}{G_i}
 &\le C_rn.\label{eq:floor-degree-energy}
\end{align}
In particular,
$\norm{G^{-1/2}(A-{ZPZ^\top})G^{-1/2}}_{\op}\le C_r/\sqrt {np}$.
Let $U$ contain the $K$ eigenvectors of $\widetilde A$ with largest absolute
eigenvalues, and put $\Pi=UU^\top$ and $H=G^{-1/2}\Pi G^{1/2}$.
The eigenvalue equation gives $\widehat A=AH$ for the reconstruction
$\widehat A=W^{-1}U\Lambda U^\top W^{-1}$ in \cref{restore:algorithm}.
The retained eigenvalues of $\widetilde A$ have magnitude at least $c_r{np}$;
the remaining eigenvalues have magnitude at most $C_r\sqrt {np}$.
There is an orthogonal matrix $O$ such that the orthonormal population
basis
\[
 V=G^{-1/2}Z(Z^\top G^{-1}Z)^{-1/2}O
\]
satisfies
\[
 \norm{U-V}_F\le C_r/\sqrt {np},
 \qquad \sqrt{G_i}\norm{V_{i\cdot}}_2\le C_r\sqrt{p}.
\]
Moreover,
\begin{equation}\label{eq:floor-block-residual}
 \norm{(I-\Pi)G^{1/2}Z}_F\le C_r\sqrt n,
\end{equation}
and every block-constant $f$ with $\norm f_\infty\le C_f$ obeys
\begin{equation}\label{eq:floor-weighted-errors}
 \norm{(I-H)f}_G\le C_rC_f\sqrt n,\qquad
 \norm{(I-H)f}_2^2\le \frac{C_rC_f^2}{p}.
\end{equation}
The same assertions hold simultaneously after deletion of any fixed
number of vertices, with each cap recomputed from the remaining graph,
after increasing the probability parameter. No raw minimum-degree
condition is imposed.
\end{proposition}

\begin{proof}
The Bernoulli transform of the total number of edges gives
$c{n^2p}\le\sum_iD_i\le C{n^2p}$ except with probability $e^{-c{n^2p}}$.
For large enough $b_0$, this event implies
\[
 c{np}\le\tau=b_0\bar D\le C{np},
 \qquad \tau\ge\max_i\sum_jP_{z_i z_j}.
\]
Since $\E A_{ij}\le p$, put
$W_{h,ii}=\sqrt{\min(1,2{np}/D_i)}$, with value one at zero degrees.
The changed set $\mathcal S=\{i:D_i>2{np}\}$ contains at most
$\frac{10}{p}$ vertices except with probability $e^{-cn}$. To check the
strength of this bound, set $m=\lfloor\frac{10}{p}\rfloor+1$. If $m\le n$
and $m$ vertices exceed the threshold, some set $T$ of size $m$ has
$\sum_{i\in T}D_i>2m{np}$. This is an independent-edge sum with
coefficients in $\{1,2\}$ and mean at most $m{np}$. At
$t=(\log2)/2$, its exponential transform bounds the tail by
\[
 \exp\{-(\log2-1/2)m{np}\},
\]
because $(e^{tc}-1)/c\le(e^{2t}-1)/2=1/2$ for $c=1,2$.
Since $\log\binom nm\le m\log(en/m)=o(m{np})$, a union over $T$
gives the stated bound. If $m>n$, the claim is automatic.

Every row of $W_hAW_h$ has squared Euclidean norm at most
$W_{h,ii}^2D_i\le2{np}$. The row-norm form of the regularized adjacency
theorem of \citet[Theorem 2.1 and Remark 2.2]{llv2017} therefore yields
\[
 \norm{W_hAW_h-\E A}_{\op}\le C_r\sqrt {np}
\]
with failure probability at most $C_rn^{-r}$. The theorem permits
these graph-dependent weights. Also
\[
 \norm{W_h{ZPZ^\top} W_h-{ZPZ^\top}}_{\op}
 \le Cp\sqrt{n|\mathcal S|}\le C\sqrt {np},
\]
by the Frobenius norm: the difference is supported on rows or columns
in $\mathcal S$. The completed-mean diagonal has norm $O(p)$.
Consequently $\norm{W_h(A-{ZPZ^\top})W_h}_{\op}\le C_r\sqrt {np}$.
The exact diagonal comparison is
\[
 G^{-1/2}=(2{np})^{-1/2}
 \diag\!\left\{\sqrt{\frac{\max(D_i,2{np})}{G_i}}\right\}W_h.
\]
The displayed diagonal multiplier is bounded, since
$G_i\ge\max(D_i,c{np})$.
Conjugation gives the stated bound for the noise in $Q$;
multiplication by $\tau\asymp {np}$ proves the adjacency-scale bound
in \eqref{eq:floor-noise}. Since $G_i\le D_i+\tau$, the total
weight $\sum_iG_i$ is at most $C{n^2p}$.

Applying the scalar-rescaled noise to $G^{1/2}\one$ gives
\[
 \norm{G^{-1/2}(A-{ZPZ^\top})\one}_2
 \le C_r{(np)}^{-1/2}\norm{G^{1/2}\one}_2\le C_r\sqrt n,
\]
which is the first sum in \eqref{eq:floor-degree-energy}.
For the second sum, $G_i-\tau=0$ when $D_i\le\tau$; otherwise
$0\le G_i-\tau=D_i-\tau\le D_i-\sum_jP_{z_i z_j}$. Thus, coordinatewise,
\[
 \frac{(G_i-\tau)^2}{G_i}
 \le\frac{(D_i-\sum_jP_{z_i z_j})^2}{G_i}.
\]
The same bound also gives
$\sum_i(1-W_{ii})^2\le \frac{C_r}{p}$, since
$(1-W_{ii})^2\le(G_i-\tau)^2/(\tau G_i)$.

Let $J=Z^\top G^{-1}Z$. For every block $k$, Cauchy--Schwarz,
comparable group sizes, $G_i\ge c{np}$ and the total-degree bound give
\[
 \frac{c}{p}\le
 \frac{n_k^2}{\sum_{i:z_i=k}G_i}
 \le J_{kk}=\sum_{i:z_i=k}G_i^{-1}\le\frac{C}{p}.
\]
The nonzero spectrum of the capped population matrix $W{ZPZ^\top} W$ is
the spectrum of $\tau J^{1/2}PJ^{1/2}$. The bound
$\sigma_K(P)\ge c p$ gives a separation of order ${np}$ from zero.
The adjacency noise bound and
subspace perturbation then give the stated eigenspace and eigenvalue
bounds. The formula for $V$
and the lower bound on $J$ give its weighted row bound.

The matrix $\tau G^{-1/2}Z$ belongs to the population signal space
$\col(WZ)=\col(G^{-1/2}Z)$. Furthermore,
\[
 \norm{G^{1/2}Z-\tau G^{-1/2}Z}_F^2
 =\sum_i\frac{(G_i-\tau)^2}{G_i}\le C_rn,
 \qquad \norm{G^{1/2}Z}_F^2\le C_r{n^2p}.
\]
Together with the projector error $O({(np)}^{-1/2})$, these imply
\eqref{eq:floor-block-residual}. The matrix $H$ is the
$G$-orthogonal projector corresponding to $\Pi$, so the first bound
in \eqref{eq:floor-weighted-errors} follows by applying the residual bound to the $K$ block values of $f$.
The second uses $G_i\ge c{np}$.

Deleting a fixed number of vertices preserves the structural bounds.
Apply the same argument to each recomputed graph and choose a larger
fixed probability parameter before the finite-depth union bound.
\end{proof}

\section{Truncated spectral-row energy}
\label{restore:tail-section}

The next lemma controls large spectral rows in expectation. Its proof
uses a recurrence at a single truncation threshold, together with
concentration of independent deleted-row innovations.

Work under the fixed structural assumptions of
\cref{prop:sparse-floor}, with $np\to\infty$ and $np^2\le C_1$.
For the high-degree cap in \cref{restore:algorithm}, recall
\[
 \tau=b_0\bar D,\qquad
 G_j=\max(D_j,\tau),\qquad
 W=\tau^{1/2}G^{-1/2},\qquad \widetilde A=WAW.
\]
Let $U$ be an orthonormal basis of the informative $K$-dimensional
eigenspace and write $\widetilde AU=U\Lambda$.
All bounds below are uniform in the fixed compact structural class.
In particular,
\begin{equation}
 0\le{P_{z_i z_j}}\le Cp,\qquad
 \norm{{(ZPZ^\top)_{i\cdot}}}_2\le C{np}/\sqrt n.
 \label{restore:tail-population-rows}
\end{equation}

Fix a probability parameter $r>0$.  Apply
\cref{prop:sparse-floor} with a sufficiently larger fixed parameter to
the full graph and all graphs obtained by deleting at most two vertices.
Let $\mathcal E_r$ be the intersection of the resulting geometry and
two-sided total-degree events.  Restricting the population matrix
preserves the fixed structural bounds and comparable block masses.
Thus
\begin{equation}
 \PP(\mathcal E_r^c)\le C_rn^{-r}.
 \label{restore:tail-geometry-probability}
\end{equation}
On this event, for every graph in the deletion family, its recomputed
threshold is comparable to ${np}$, its total degree is at most $C_r{n^2p}$,
its centered weighted noise has norm at most $C_r\sqrt {np}$, its
informative eigenvalues have absolute values at least $c_r{np}$, and its
remaining eigenvalues have absolute values at most $C_r\sqrt {np}$.
For a deleted training graph, use the analogous event formed from
that graph and its further deletions. This event is measurable from
the training graph alone and has the same probability bound.

The following deterministic properties of the cap will be useful:
\begin{equation}
 W_j\le1,\qquad W_j^2D_j\le\tau,
 \qquad \norm{\widetilde A}_{\op}\le\tau.
 \label{restore:tail-floor-operator}
\end{equation}
Indeed, $G\ge\diag(A\one)$ and the inequality
$2|ab|\le a^2+b^2$ gives
$|x^\top Qx|\le\sum_j(D_j/G_j)x_j^2\le\norm x_2^2$,
where $Q=\widetilde A/\tau$.

\begin{lemma}[Truncated energy for the capped-adjacency eigenvectors]
\label{restore:tail-energy-lemma}
There are fixed constants $R_0,C,c>0$ such that, for every fixed
$R\ge R_0$ and all sufficiently large $n$,
\begin{equation}
 \E\!\left[\ind{\mathcal E_r}
       \sum_{j=1}^n\norm{U_{j\cdot}}_2^2
                  \ind{\sqrt n\norm{U_{j\cdot}}_2>R}\right]
 \le CR^2e^{-c{np}\log R}.
 \label{restore:tail-energy}
\end{equation}
The constants $R_0,C,c$ depend on the fixed structural bounds,
$b_0$, and $r$, but not on $R$ or $n$.  The sufficiently-large-$n$
condition may depend on the chosen fixed $R$.  The same result holds
for each recomputed deleted graph with its training-measurable
geometry event.
\end{lemma}

\begin{proof}
We give the weight perturbation, the same-threshold energy recurrence,
and its expectation bounds separately.

\medskip\noindent\textbf{Recomputed weights and the eigenspace residual.}
For $J=[n]\setminus\{i\}$, let $W^{(-i)}$ and $U^{(-i)}$ be
computed from $A_{JJ}$, with $U^{(-i)}$ embedded by a zero row at $i$.
Outside the geometry event choose an arbitrary deterministic
orthonormal fallback depending only on $A_{JJ}$.
For an all-zero training graph set $W^{(-i)}=I$ and use that
orthonormal fallback. Thus the proof-only completion always has
$\norm{W^{(-i)}}_{\op}\le1$, and the independent innovation below
is defined for every training graph.
Put $(\Gamma_i)_{jj}=W_j^{(-i)}/W_j-1$ for $j\in J$ and
$(\Gamma_i)_{ii}=0$.
The exact threshold difference is
\[
 \tau^{(-i)}-\tau
 =b_0\frac{\bar D-2D_i}{n-1}.
\]
Comparing the two square-root caps, using the comparability of the
thresholds to ${np}$, gives on $\mathcal E_r$
\begin{equation}
 |(\Gamma_i)_{jj}|
 \le C\left\{\frac{A_{ij}}{G_j}
             +\frac{{np}+D_i}{{n^2p}}\right\},
 \qquad \norm{\Gamma_i}_{\op}\le C/(np).
 \label{restore:tail-weight-change}
\end{equation}
This also covers a crossing of the cap threshold. It follows either
by splitting the cases in the two maxima or by comparing the logarithms
of $\sqrt{\tau/\max(D_j,\tau)}$; the relative threshold change is
$O(({np}+D_i)/({n^2p}))$.  The last bound uses $D_i\le n$ and ${np}\le Cn$.

There is an orthogonal matrix $O_i$ such that
\begin{equation}
 \norm{U-U^{(-i)}O_i}_F
 \le C\norm{U_{i\cdot}}_2+C\norm{\Gamma_iU}_F.
 \label{restore:tail-loo-residual}
\end{equation}
To verify this, write
$\widetilde A=\left(\begin{smallmatrix}0&m_i^\top\\m_i&B_i\end{smallmatrix}\right)$.
By \eqref{restore:tail-floor-operator},
$\norm{m_i}_2\le\sqrt\tau$ and $\norm{B_i}_{\op}\le\tau$.
Also $B_iU_J=U_J\Lambda-m_iU_{i\cdot}$.
The lower block of the recomputed deleted matrix is
$(I+\Gamma_i)B_i(I+\Gamma_i)$.
Subtracting the full matrix and applying the difference to $U$ gives
a Frobenius residual bounded by
$C{np}(\norm{U_{i\cdot}}_2+\norm{\Gamma_iU}_F)$.
The full informative spectrum is separated from the deleted
noninformative spectrum by $c{np}$.
The complementary-space Sylvester equation and orthogonal Procrustes
alignment give \eqref{restore:tail-loo-residual}.
Only separation from the complementary spectrum is used; the
informative eigenvalues can repeat or have either sign.

\medskip\noindent\textbf{Truncation and summable weight changes.}
Set $t=R/\sqrt n$.  Let $f_t$ be radial projection onto the Euclidean
ball of radius $t$, applied to each row, and put
\[
 \bar U=f_t(U),\qquad E_U=U-\bar U,\qquad
 \mathcal T_R(U)=\sum_j\norm{U_{j\cdot}}_2^2
                         \ind{\norm{U_{j\cdot}}_2>t}.
\]
The radial projection is $1$-Lipschitz and commutes with orthogonal
rotations.  Moreover $\norm{E_U}_F^2\le \mathcal T_R(U)$,
$\max_j\norm{\bar U_{j\cdot}}_2\le t$, and $\norm{\bar U}_F\le\sqrt K$.
Define
\[
 \xi_i=W_i\left(\sqrt{D_i}+C{np}/\sqrt n\right).
\]
By \eqref{restore:tail-floor-operator}, $\xi_i^2\le C{np}$.
Equation \eqref{restore:tail-weight-change} yields
\[
 \norm{\Gamma_i\bar U}_F
 \le Ct\frac{\sqrt{D_i}}{np}+C\frac{{np}+D_i}{{n^2p}},
\]
and therefore
\begin{equation}
 \frac{C\xi_i}{np}\norm{\Gamma_i\bar U}_F
 \le Ct\left\{\frac1{np}+\frac{\sqrt{D_i}}{{(np)}^{3/2}}\right\}.
 \label{restore:tail-bounded-weight}
\end{equation}
For the global term, use $t\ge n^{-1/2}$, $D_i\le n$, and
${np}\le Cn$: $D_i/(n{(np)}^{3/2})\le t\sqrt{D_i}/{(np)}^{3/2}$ and
$1/(n\sqrt {np})\le Ct/(np)$.
The square-root cap does not give $W_iD_i\le C{np}$. The
right side of \eqref{restore:tail-bounded-weight} can exceed a
sufficiently small fixed multiple of $t$, for large ${np}$, only if
$D_i\ge c_*{(np)}^3$.

The unbounded part satisfies the summable estimate
\begin{equation}
 \sum_i\xi_i^2\norm{\Gamma_iE_U}_F^2\le C\norm{E_U}_F^2.
 \label{restore:tail-summable-weight}
\end{equation}
Indeed, after interchanging $i,j$, the local coefficient of
$\norm{(E_U)_{j\cdot}}_2^2$ is at most
$C{np}D_j/G_j^2\le C$.
The global coefficient is bounded by
\[
 \frac{C}{n^2{np}}\sum_i({np}+D_i)^2\le C,
\]
because $\sum_iD_i\le C{n^2p}$ and
$\sum_iD_i^2\le n\sum_iD_i$.
These estimates do not use a maximum-degree event.

\medskip\noindent\textbf{Independent innovations and the recurrence.}
Define the unrotated, training-measurable-coefficient innovation
\begin{equation}
 K_i=(A_{iJ}-{(ZPZ^\top)_{iJ}})W^{(-i)}f_t(U_J^{(-i)}).
 \label{restore:tail-innovation}
\end{equation}
Let $E_0=W(A-{ZPZ^\top})W$.
The eigenvector equation is
\[
 U=W{ZPZ^\top} WU\Lambda^{-1}+E_0\bar U\Lambda^{-1}+E_0E_U\Lambda^{-1}.
\]
The first term has every row bounded by $C/\sqrt n$, by
\eqref{restore:tail-population-rows}.
For the second, compare $W_J\bar U_J$ to
$W^{(-i)}f_t(U_J^{(-i)})O_i$ and use radial Lipschitzness,
\eqref{restore:tail-loo-residual}, and
\eqref{restore:tail-weight-change}.  This gives
\begin{align}
 \norm{U_{i\cdot}}_2\le{}&
 C/\sqrt n+(C/(np))\norm{K_i}_2
 +(C\xi_i/(np))\norm{U_{i\cdot}}_2\nonumber\\
 &+(C\xi_i/(np))\norm{\Gamma_i\bar U}_F
 +(C\xi_i/(np))\norm{\Gamma_iE_U}_F
 +(C/(np))\norm{(E_0E_U)_{i\cdot}}_2.
 \label{restore:tail-row-recursion}
\end{align}
The completed-mean diagonal contributes at most $Cpt/(np)\le Ct/n$
and is absorbed in the bounded terms; it is not treated as an
observed self-loop.
Although $O_i$ may depend on the tested row, the argument only uses
$\norm{K_iO_i}_2=\norm{K_i}_2$.
No conditional concentration is applied to rotated coordinates.

Take $R$ sufficiently large and then ${np}$ sufficiently large to absorb
the baseline and self terms.  Retain from
\eqref{restore:tail-bounded-weight} the excess
$Ct\sqrt{D_i/{(np)}^3}\ind{D_i\ge c_*{(np)}^3}$.
For $\norm{U_{i\cdot}}_2>t$, separate the cases
$\norm{K_i}_2>c'{np}t$ and $\norm{K_i}_2\le c'{np}t$ and square.
The resulting bound is
\begin{align}
 \norm{U_{i\cdot}}_2^2\ind{\norm{U_{i\cdot}}_2>t}
 \le{}&\frac C{{(np)}^2}\norm{K_i}_2^2
                         \ind{\norm{K_i}_2>c'{np}t}\nonumber\\
 &+\frac C{{(np)}^2}\norm{(E_0E_U)_{i\cdot}}_2^2
 +\frac{C\xi_i^2}{{(np)}^2}\norm{\Gamma_iE_U}_F^2
 +\frac{Ct^2}{{(np)}^3}D_i\ind{D_i\ge c_*{(np)}^3}.
 \label{restore:tail-pointwise-recursion}
\end{align}
Summing, using $\norm{E_0}_{\op}^2\le C{np}$ and
\eqref{restore:tail-summable-weight}, and multiplying by
$\ind{\mathcal E_r}$ gives
\begin{align}
 \mathcal T_R(U)\ind{\mathcal E_r}
 \le{}&\frac C{{(np)}^2}\sum_i\norm{K_i}_2^2
                         \ind{\norm{K_i}_2>c'{np}t}
 +\frac C{np}\mathcal T_R(U)\ind{\mathcal E_r}\nonumber\\
 &+\frac{Ct^2}{{(np)}^3}\sum_iD_i\ind{D_i\ge c_*{(np)}^3}.
 \label{restore:tail-energy-recursion}
\end{align}
Both energy terms use the same threshold $R$.

\medskip\noindent\textbf{Expectation bounds.}
Conditional on $A_{JJ}$, \eqref{restore:tail-innovation} is a sum
of independent centered vectors.  In each coordinate its summands
are bounded by $t$, and its total variance is at most
\[
 Cp\norm{W^{(-i)}f_t(U^{(-i)})}_F^2\le Cp.
\]
These bounds hold for every training graph, without conditioning on
$\mathcal E_r$.  Bennett's inequality in the fixed $K$ coordinates,
followed by integration of the squared-norm tail, gives
\begin{equation}
 \E\!\left[\norm{K_i}_2^2
                        \ind{\norm{K_i}_2>c'{np}t}\right]
 \le C{(np)}^2t^2e^{-c_1{np}\log R}.
 \label{restore:tail-innovation-moment}
\end{equation}
To see the exponent, set $v=Cp$ and
$h(u)=(1+u)\log(1+u)-u$.
At $x=c'{np}t$, the Bennett exponent
$(v/t^2)h(tx/v)$ is at least $c_1{np}\log R$, since
$tx/v$ is proportional to $R^2$.
Beyond that point the derivative of the exponent is at least
$c\log R/t$, which gives the same bound upon tail integration.

Every $D_i$ is a Bernoulli sum with mean at most $C{np}$.
Its Chernoff bound and tail integration yield
\begin{equation}
 \E[D_i\ind{D_i\ge c_*{(np)}^3}]
 \le C{(np)}^3e^{-c_2{(np)}^3\log {np}}.
 \label{restore:tail-large-degree-moment}
\end{equation}
The left side is zero if $c_*{(np)}^3>n$.
Taking expectations in \eqref{restore:tail-energy-recursion},
absorbing $C/(np)$, and using $nt^2=R^2$ now gives
\[
 \E[\mathcal T_R(U)\ind{\mathcal E_r}]
 \le CR^2\{e^{-c_1{np}\log R}+e^{-c_2{(np)}^3\log {np}}\}.
\]
For each fixed $R$, the second exponential is eventually bounded by
the first after reducing $c_1$.  This proves
\eqref{restore:tail-energy}.  The proof is uniform under deletion
of a fixed number of vertices and uses their recomputed weights,
so it also proves the asserted training-graph version.
The same proof applies to any event supplying the stated aggregate
geometry, including the exponentially accurate event of
\cref{restore:high-aggregate}. The innovation and degree-tail bounds
are unconditional before restriction to that event.
\end{proof}

\begin{corollary}[Exceptional coefficients are hit with exponentially small mass]
\label{restore:tail-hit-corollary}
On a training graph with the preceding cap, put
$H=G^{-1/2}UU^\top G^{1/2}$.
Let $f$ be any training-measurable vector satisfying
$\norm f_\infty\le C_f$, and let $\varepsilon=(H-I)f$.
For every prescribed fixed $s_0>0$, there is a fixed $L_\varepsilon<\infty$
such that, for any independent tested Bernoulli row with probabilities
$p_j\le Cp$,
\begin{equation}
 \E\!\left[\ind{\mathcal E_r}
                 \sum_{j:|\varepsilon_j|>L_\varepsilon}p_j\right]
 \le C_{s_0}e^{-s_0{np}}.
 \label{restore:tail-hit-mass}
\end{equation}
Here $\mathcal E_r$ is the training-measurable geometry event.
\end{corollary}

\begin{proof}
On that event, $G_j\ge c{np}$ and $\sum_jG_j\le C{n^2p}$.
For a fixed constant $C$ determined by these two bounds,
\[
 |\varepsilon_j|
 \le C_f+\frac{\norm{U_{j\cdot}}_2}{\sqrt{G_j}}
                  \norm{U^\top G^{1/2}f}_2
 \le C_f+CC_f\sqrt n\norm{U_{j\cdot}}_2.
\]
For $L_\varepsilon=C_f+CC_fR$, the event
$|\varepsilon_j|>L_\varepsilon$ implies $\sqrt n\norm{U_{j\cdot}}_2>R$.
Thus \cref{restore:tail-energy-lemma} gives
\begin{equation}
 \E\!\left[\ind{\mathcal E_r}
                    \sum_j\ind{|\varepsilon_j|>L_\varepsilon}\right]
 \le\frac n{R^2}\E[\mathcal T_R(U)\ind{\mathcal E_r}]
 \le Cne^{-c{np}\log R}.
 \label{restore:tail-coefficient-count}
\end{equation}
Multiplication by $Cp$ bounds the expected hit mass by
$C{np} e^{-c{np}\log R}$.
Choose a fixed $R$ with $c\log R>s_0+1$ and absorb the factor ${np}$.
Independence is required only to interpret the conditional probability
that the tested row hits the training exceptional set; no independence
between the training geometry event and its eigenvectors is asserted.
\end{proof}

\begin{corollary}[Average exceptional target-row mass]
\label{restore:tail-target-corollary}
For every fixed degree cutoff $L_D$ and every fixed $s_0>0$, there is
a fixed $L_0$ such that
\begin{equation}
 \frac1n\sum_i\PP\!\left\{
    \sqrt{G_i}\norm{U_{i\cdot}}_2>L_0\sqrt{p},\quad
    D_i\le L_D{np},\quad \mathcal E_r\right\}
 \le C_{s_0,L_D}e^{-s_0{np}}.
 \label{restore:tail-target-mass}
\end{equation}
\end{corollary}

\begin{proof}
On the displayed event, $G_i\le C_{L_D}{np}$.
Choose $L_0=R\sqrt{C_{L_D}}$.
The target-row exception then implies
$\sqrt n\norm{U_{i\cdot}}_2>R$.
Apply \eqref{restore:tail-energy} and divide its corresponding
expected count by $n$.  Choosing fixed $R$ sufficiently large gives
\eqref{restore:tail-target-mass}.
\end{proof}

\section{Local stability under deletion of the tested vertex}
\label{restore:loo}

The comparison below requires spectral row bounds only on neighbors
of the tested vertex. The cap threshold is recomputed after deletion.
We divide each capped adjacency matrix by its own scalar threshold
to write a generalized eigenvalue equation; this does not change its
retained eigenspace or the reconstruction $AH$.

\begin{lemma}[A local scalar comparison for capped adjacency]\label{restore:loo-scalar}
Put the tested vertex $i$ first and write $A=\left(\begin{smallmatrix}0&x^\top\\x&B\end{smallmatrix}\right)$,
$D_i=\sum_jx_j\le L_D{np}$, and let
$\Delta=\diag\{\max(D_{B,j},\tau_B)\}$ be recomputed from $B$.
Suppose
\[
 c{np}\le\Delta_j,\qquad \sum_j\Delta_j\le C{n^2p},
 \qquad Q_B=\Delta^{-1/2}B\Delta^{-1/2}
   =U_B(\Lambda_B/\tau_B)U_B^\top+N,
\]
where $\Lambda_B$ contains the retained eigenvalues of $\tau_BQ_B$,
$NU_B=0$, $\norm N_{\op}\le C/\sqrt {np}$, and the $K$ retained
eigenvalues of $Q_B$ have a fixed gap from zero. Let the full diagonal cap weight be
$G=\diag(G_i,\Delta+E)$, with
\[
 c{np}\le G_i\le C{np},\qquad E=E_\mathrm{loc}+E_\mathrm{glob},\qquad
 0\le E_\mathrm{loc}\le\diag(x),\quad
 \norm{E_\mathrm{glob}}_{\op}\le Cp.
\]
Here ${np}\to\infty$ and ${np}\le Cn$. Suppose only on hit coordinates that
\begin{equation}\label{restore:loo-hit}
 \norm{\Delta_j^{-1/2}(U_B)_{j\cdot}}_2
       \le \frac{R}{n\sqrt p}\quad\text{whenever }x_j=1.
\end{equation}
Let $U$ contain orthonormal informative eigenvectors of $\widetilde A=WAW$, and put
\[
 H=G^{-1/2}UU^\top G^{1/2},\qquad
 H_B=\Delta^{-1/2}U_BU_B^\top\Delta^{1/2}.
\]
For every $\norm w_\infty\le C_w$, uniformly over the stated inputs,
\begin{equation}\label{restore:loo-main}
 (AHw)_i=x^\top H_Bw_{-i}+O(1).
\end{equation}
Implicit constants may depend on $L_D,R,C_w$ and the other fixed bounds
in the statement. No bound is assumed on unhit rows of $U_B$ or on $\max_j\Delta_j$.
\end{lemma}

\begin{proof}
For $D_i=0$ both row actions vanish. For $D_i>0$, put
$E_\Delta=\Delta^{-1/2}E\Delta^{-1/2}$, $h=\Delta^{-1/2}x$,
$\Pi_B^\perp=I-U_BU_B^\top$. Then $\norm{E_\Delta}\le C/(np)$ and
$\norm h^2\le CD_i/(np)$. The scalar-rescaled matrix $Q=G^{-1/2}AG^{-1/2}$ differs from
$\diag(0,Q_B)$ by $O({(np)}^{-1/2})$: its off-diagonal row has norm
at most $C\sqrt{D_i/(G_i{np})}$, and the lower renormalization has norm
$O({(np)}^{-1})$. Thus the full signal remains separated.

For an eigenvalue $\mu_r$ of $Q=\widetilde A/\tau_A$, write a full
generalized eigenvector as
\[
 G^{-1/2}u_r=\binom{t_r}{\Delta^{-1/2}(U_By_r+u_r^\perp)},
 \qquad U_B^\top u_r^\perp=0.
\]
The exact top and complementary equations are
\begin{align*}
 h^\top(U_By_r+u_r^\perp)&=\mu_r G_i t_r,\\
 (\mu_rI-N+\mu_r\Pi_B^\perp E_\Delta \Pi_B^\perp)u_r^\perp
   &=t_r\Pi_B^\perp h-\mu_r\Pi_B^\perp E_\Delta U_By_r.
\end{align*}
The inverse $T_r$ of the complementary operator is bounded. Therefore
\[
 (\mu_r G_i-(\Pi_B^\perp h)^\top T_r\Pi_B^\perp h)t_r
  =(U_B^\top h)^\top y_r
       -\mu_r(\Pi_B^\perp h)^\top T_r\Pi_B^\perp E_\Delta U_By_r.
\]
The denominator has magnitude at least $c{np}-CD_i/(np)\ge c'{np}$.
Here $G_i=\max(D_i,\tau_A)\asymp {np}$ is the diagonal weight in this formula.
The matrix $Y=(y_1,\ldots,y_K)$ and its inverse have bounded norm,
by the signal-space angle and $(I+E_\Delta)^{1/2}=I+O({(np)}^{-1})$.

The local row premise and the global threshold perturbation give
\[
 \norm{U_B^\top h}\le \frac{C_RD_i}{n\sqrt p},\qquad
 \norm{E_\Delta U_B}_F\le C_R\sqrt{D_i}/({np}\sqrt n)+C/n.
\]
The $C/n$ term is retained; it need not be supported on the neighbors.
Substitution yields, for $U_\perp=(u_1^\perp,\ldots,u_K^\perp)$,
\begin{equation}\label{restore:loo-refined}
 \norm t\le \frac{C_R}{n\sqrt p},\qquad
 \norm{U_\perp}\le C_R\{\sqrt{D_i}/({np}\sqrt n)+1/n\}.
\end{equation}
Let $V=\Delta^{-1/2}U_B$, $v_r=\Delta^{-1/2}u_r^\perp$ and
$f_r=Vy_r+v_r$. Write $t=(t_1,\ldots,t_K)^\top$, 
$v=(v_1,\ldots,v_K)$ and $f=(f_1,\ldots,f_K)$. Exact generalized orthonormality gives
\[
 Y^\top Y+U_\perp^\top U_\perp+G_i tt^\top+f^\top Ef=I.
\]
On the local part of $E$, use the hit-row bound; on its global part
use $\norm{E_\mathrm{glob}}\le Cp$ and
$\norm f\le C/\sqrt {np}$. Consequently
\begin{equation}\label{restore:loo-gram}
 \norm{YY^\top-I}\le C_R\{D_i/({n^2p})+1/n\}.
\end{equation}

For $w_V=V^\top\Delta w_{-i}$, the exact row expansion is
\[
 (AHw)_i=
 \sum_r(x^\top Vy_r+x^\top v_r)
 \{G_i t_rw_i+y_r^\top w_V+v_r^\top\Delta w_{-i}
                                      +f_r^\top Ew_{-i}\}.
\]
Only total degree is needed to bound global coefficient vectors. In
particular,
\begin{align*}
 \norm{x^\top V}&\le \frac{C_RD_i}{n\sqrt p},&
 \norm{w_V}&\le Cn\sqrt p,\\
 \norm{x^\top v}&\le C_R\{D_i/({(np)}^{3/2}\sqrt n)+\sqrt{D_i}/(n\sqrt {np})\},&
 \norm{v^\top\Delta w_{-i}}&\le C_R\{\sqrt{D_i/(np)}+\sqrt{p}\},\\
 \norm{G_i t}&\le C_R\sqrt{p},&
 \norm{f^\top Ew_{-i}}&\le C_R\{\frac{D_i}{n\sqrt p}+\sqrt{p}\}.
\end{align*}
For example the global part of the last bound follows from
$|f_r^\top E_\mathrm{glob}w_{-i}|
\le (Cp)C_w\sqrt n\norm{f_r}\le C\sqrt{p}$.
The leading term is $x^\top VYY^\top w_V$; its error relative to
$x^\top VV^\top\Delta w_{-i}$ is
$O(D_i^2/({n^2p})+D_i/n)$. Every other product is bounded by a fixed
multiple of
\[
 D_i/(np)+\sqrt{D_i/n}+D_i^2/({n^2p})+D_i/n+D_i\sqrt{D_i}/({np}\sqrt n).
\]
These quantities are bounded for $D_i\le L_D{np}$ and ${np}\le Cn$.
This proves the scalar comparison.
\end{proof}

For the empirical cap threshold $\tau=b_0\bar D$, the displayed decomposition of
$E$ is automatic on the two-sided total-degree event. Indeed, if
$\bar D_B$ is the mean degree of $B$, then
\[
 \tau_A-\tau_B=b_0(2D_i-\bar D_B)/n.
\]
First change each deleted degree by $x_j$ while keeping $\tau_B$;
this yields $E_\mathrm{loc}\in[0,\diag(x)]$. Then change the threshold;
this yields $\norm{E_\mathrm{glob}}\le C_{L_D}p$.

\begin{corollary}[Exact-degree reconstruction and simple hollowing]
\label{restore:loo-corrected}
Let $\widehat A^{\mathrm{c}}$ be the exact-degree hollow reconstruction in
\eqref{restore:reconstruction}. Then $\widehat A^{\mathrm{c}}\one=A\one$ and
$\diag(\widehat A^{\mathrm{c}})=0$ exactly.
For $\bar w_{-i}=(n-1)^{-1}\sum_{j\ne i}w_j$, the same local premises
give
\begin{equation}\label{restore:loo-corrected-action}
 (\widehat A^{\mathrm{c}} w)_i
 =x^\top\{H_B(w_{-i}-\bar w_{-i}\one)+\bar w_{-i}\one\}
       +O(1).
\end{equation}
\end{corollary}

\begin{proof}
The definition gives the exact row identity
\[
 (\widehat A^{\mathrm{c}} w)_i
 =[\widehat A(w-\bar w_{-i}\one)]_i+D_i\bar w_{-i}
       -\widehat A_{ii}(w_i-\bar w_{-i}).
\]
Apply the preceding lemma to $w-\bar w_{-i}\one$, whose entries
are bounded by $2C_w$. Equation~\eqref{restore:loo-refined} gives
$\norm{U_{i\cdot}}_2^2=G_i\norm t^2\le C_R/n$, so
$|\widehat A_{ii}|\le C_Rp$. The final term is therefore $O(p)=O(1)$,
proving \eqref{restore:loo-corrected-action}. The degree and diagonal
identities follow directly from \eqref{restore:reconstruction}.
\end{proof}

The independent coefficient vector in
\eqref{restore:loo-corrected-action} has perturbation
$(H_B-I)(w_{-i}-\bar w_{-i}\one)$. For block-constant contrasts it
therefore admits the tilted cancellation in
\cref{lem:sbm-floor-cancellation}.
For clarity, the hit-row exception is paid directly. On the
training event, $\Delta_j\ge c {np}$, so
\[
 \{j:\norm{\Delta_j^{-1/2}(U_B)_{j\cdot}}_2
                          >\frac{R}{n\sqrt p}\}
 \subseteq
 \{j:\sqrt n\norm{(U_B)_{j\cdot}}_2>R\sqrt c\}.
\]
\Cref{restore:tail-energy-lemma} bounds the expected cardinality of
this set on the training event by $Cn e^{-c'{np}\log R}$.
Conditioning only on the deleted graph and using $P_{z_i z_j}\le p$
therefore bounds the probability that the tested row hits it by
$C{np} e^{-c'{np}\log R}$. Choosing a sufficiently large fixed $R$
gives any prescribed fixed exponential rate $s$.
The upper-degree exception $D_i>L_D{np}$ has arbitrarily large exponential
rate after increasing the fixed $L_D$. Neither bound uses a union over
all possible tested targets.

Uniformity in $w$ is an algebraic stability statement. It does not make
full-graph fitted coefficients independent of $x$. The fitted-score
bound first replaces them by deterministic true contrasts; only then
is the independent sharp-tail bound applied.

\section{Block fitting and misclassification control}
\label{restore:fit}

Write $\Pi=UU^\top$. For the
proof, rescale the encoded adjacency matrix by its common threshold:
$Q=\widetilde A/\tau=G^{-1/2}AG^{-1/2}$. This changes the eigenvalue scale,
not the retained eigenspace. We work on the aggregate event
\begin{align}
 \norm{Q-G^{-1/2}{ZPZ^\top} G^{-1/2}}_{\op}
   &\le C/\sqrt {np},& \sum_iG_i&\le C{n^2p},
 \label{restore:fit-geometry-one}\\
 \sum_i\frac{(G_i-\tau)^2+(D_i-\sum_jP_{z_i z_j})^2}{G_i}
   &\le Cn,&
 \norm{(I-\Pi)G^{1/2}Z}_F&\le C\sqrt n,
 \label{restore:fit-geometry-two}\\
 \max_{k,\ell}\left|[Z^\top(A-{ZPZ^\top})Z]_{k\ell}\right|
   &\le Cn.&&\label{restore:fit-block-event}
\end{align}
We also have $G_i\ge\tau\asymp {np}$, $\sum_jP_{z_i z_j}\asymp {np}$,
$\bar D\asymp {np}$, and $\tau\ge\max_i\sum_jP_{z_i z_j}$. The first two
lines follow from \cref{prop:sparse-floor}; the block-total event is
paid in the risk proof. All bounds in this section are deterministic
on this event. Constants depend only on the fixed structural and
aggregate-event bounds and, for local estimates, on the fixed thresholds
$L_D,L_U$. They do not depend on the fitted partition. As in the
parameter-estimation and classification arguments of \citet{zhouli}, we display the dependence on misclassification rather
than absorbing it into a separately named error-radius constant.
Align the current labels once with the true labels. Then the incorrectly
assigned set $E=\{i:\widehat z_i\ne z_i\}$ has
$|E|=n\Mis(\widehat z,z)$. The iteration argument below preserves this same label matching.
The matching may depend on the entire graph. Every deterministic bound
below holds simultaneously over the aligned partitions; no probability
calculation conditions on the matching or on the estimated partition.
After relabeling, a tied update is still a maximizer, which suffices for
the non-strict pairwise error inclusion.
For current labels $\widehat z$, let
$\widehat Z_{ik}=\ind{\widehat z_i=k}$, and use
$n_k(\widehat z)$ for their community sizes.

An aligned population basis is
\[
 V=G^{-1/2}Z(Z^\top G^{-1}Z)^{-1/2}O,
\]
where $O$ is orthogonal. The signal gap gives
\begin{equation}\label{restore:fit-basis}
 \norm{U-V}_F\le C/\sqrt {np},\qquad
 \sqrt{G_i}\norm{V_{i\cdot}}_2\le C\sqrt{p}.
\end{equation}
Moreover $\norm Q_{\op}\le1$, since
$|x^\top Ax|\le\sum_iD_ix_i^2\le\sum_iG_ix_i^2$.
In particular $\norm{\Lambda/\tau}_{\op}\le1$, and
$\widehat A=G^{1/2}Q_KG^{1/2}$ for $Q_K=U(\Lambda/\tau)U^\top$.

For the row bounds in this section, put
$a_i=\sqrt{G_i}\norm{U_{i\cdot}}_2$.

\begin{lemma}[Weighted mass of incorrectly assigned vertices]
\label{restore:fit-wrong-mass}
For every set $E\subseteq[n]$,
\begin{equation}\label{restore:fit-wrong-mass-bounds}
 \sum_{i\in E}G_i\le C(n+np|E|),\qquad
 \sum_{i\in E}a_i\le C\left(|E|\sqrt p+\frac1{\sqrt p}\right).
\end{equation}
Also $\sum_i a_i\le Cn\sqrt p$. These bounds hold simultaneously
for all sets; in particular, the error set has
$|E|=n\Mis(\widehat z,z)$ under the common label matching.
\end{lemma}

\begin{proof}
The inequality $G_i\le2\tau+(G_i-\tau)^2/G_i$ and the aggregate
energy bound give $\sum_EG_i\le C(np|E|+n)$. By
\eqref{restore:fit-basis} and Cauchy--Schwarz,
\[
 \sum_{i\in E}a_i
 \le C|E|\sqrt p+
      \left(\sum_{i\in E}G_i\right)^{1/2}\norm{U-V}_F
 \le C\left(|E|\sqrt p+\sqrt{|E|+p^{-1}}\right).
\]
The inequality $\sqrt{p|E|+1}\le p|E|+1$ yields the second bound.
The global bound follows from
$\sum_i a_i\le(\sum_iG_i)^{1/2}\norm U_F$ and fixed $K$.
\end{proof}

\begin{lemma}[Normalized spectral geometry]\label{lem:normalized-geometry}
On the aggregate event, let $x_i$ be the normalized spectral rows used
in Section~\ref{sec:initialization}. The orthogonal matrix $O$ in
\eqref{restore:fit-basis} satisfies
\[
 \sum_i\norm{x_i-O_{z_i\cdot}}_2^2\lesssim \frac1p,
 \qquad \norm{O_{k\cdot}-O_{\ell\cdot}}_2=\sqrt2\quad(k\ne\ell).
\]
\end{lemma}

\begin{proof}
The normalized rows of $V$ are the $K$ rows of $O$, with pairwise
distance $\sqrt2$. Choose a fixed $B_0$ larger than twice the
uniform upper bound on $\tau/(np)$. Each row with $G_i>B_0{np}$
contributes at least $B_0{np}/4$ to $\sum_i(G_i-\tau)^2/G_i$, so
there are at most $\frac{C}{p}$ such rows. For every remaining row,
$\norm{V_{i\cdot}}_2\ge c/\sqrt n$. The normalization inequality
\[
 \left\|\frac u{\norm u_2}-\frac v{\norm v_2}\right\|_2
 \le\frac{2\norm{u-v}_2}{\norm v_2}
\]
(with any fixed unit direction when $u=0$) and
\eqref{restore:fit-basis} imply that the normalized empirical rows
$x_i$ have total squared distance at most $\frac{C}{p}$ from their true
centers $O_{z_i\cdot}$: the remaining rows satisfy this by the displayed
inequality, and the $O(\frac{1}{p})$ high-degree rows each contribute at most four.

\end{proof}

\begin{lemma}[$k$-means++ accuracy conditional on the graph]
\label{lem:kmpp-initializer}
On the aggregate event there is a fixed $C_0<\infty$ such that the
row-normalized $k$-means++ initialization of
Section~\ref{sec:initialization}, with $m\ge1$ independent restarts,
satisfies
\begin{equation}\label{eq:kmpp-conditional}
 \PP\!\left\{\Mis(\widehat z^{(0)},z)>\frac{C_0}{np}
                  \,\middle|\,A\right\}\le 2^{-m}
\end{equation}
for all sufficiently large $n$. Only the independent seeding randomness
is averaged in this conditional probability. The result still holds
with a graph-measurable restart count $m=m(A)$ and the permitted
nonincreasing Lloyd steps.
\end{lemma}

\begin{proof}
For fixed normalized points, write $\Phi$ for the within-cluster sum
of squared distances and $\Phi_{\mathrm{opt}}$ for its minimum over
$K$ centers. These are local proof quantities, not new model parameters.
The true centers in \cref{lem:normalized-geometry} give
$\Phi_{\mathrm{opt}}\lesssim 1/p$. The seeding theorem of
\citet[Theorem 1.1]{arthur2007} states
\[
 \E(\Phi_{\mathrm{seed}}\mid A)
       \le 8(\log K+2)\Phi_{\mathrm{opt}}.
\]
The stated Lloyd steps do not increase this objective. Markov's
inequality consequently gives at least probability $1/2$ that one
run has $\Phi\le16(\log K+2)\Phi_{\mathrm{opt}}$.
If $\Phi_{\mathrm{opt}}=0$, the objective is zero almost surely and
the same conclusion holds. Conditional independence of the runs
implies that the smallest observed cost meets this bound except
with probability at most $2^{-m}$.

For completeness, convert this objective bound into a label bound.
Let $\mu_1,\ldots,\mu_K$ be the retained fitted centers and let
$\widehat z^{(0)}$ be their nearest-center labels. Since $K$ is fixed,
a successful run has
\[
 \sum_i\norm{\mu_{\widehat z^{(0)}_i}-O_{z_i\cdot}}_2^2
 \le 2\Phi+2\sum_i\norm{x_i-O_{z_i\cdot}}_2^2
 \lesssim\frac1p.
\]
There are therefore at most $C/p$ vertices for which the fitted center
is at distance at least $\sqrt2/3$ from the correct population center.
Every true group has at least one remaining vertex for large $np$,
since its size is at least $n/(\beta K)$. The $K$ balls of radius
$\sqrt2/3$ about the population centers are disjoint. Thus the $K$
fitted center indices are in bijection with the $K$ true groups;
all vertices outside the exceptional set have the correct matched
label. In particular, the successful run has no empty fitted group.
Dividing by $n$ proves \eqref{eq:kmpp-conditional}.
This is the approximate-clustering-to-misclassification argument
underlying \citet[Corollary 1]{zhouamini2019}, specialized here to
normalized centers at separation $\sqrt2$.
\end{proof}

\begin{lemma}[Radius-count initialization in the archived experiments]
\label{restore:fit-initializer}
The radius-count rule specified in Section~\ref{sec:experiments},
applied to the same regularized $U$, satisfies
\[
 \Mis(\widehat z,z)\lesssim\frac1{np}
\]
on the aggregate event. It is an alternative initialization, not the
$k$-means++ implementation of \cref{cor:kmpp-oracle}.
\end{lemma}

\begin{proof}
Use the normalized rows $x_i$ and population centers of
\cref{lem:normalized-geometry}.
Call a row good if $\norm{x_i-O_{z_i\cdot}}_2\le1/32$.
At most $m=\frac{C}{p}$ rows are bad, and $m<n/(3\beta K)$ for large $n$.
Inductively, all good rows of a selected block have been removed and
all good rows of unselected blocks remain. A radius-$1/8$ ball
centered at a good row of an unselected block contains every good row
of that block, hence more than $m$ remaining rows. The maximizing
ball must therefore contain a good row from an unselected block;
its representative is within $1/8+1/32=5/32$ of that center.
Every good row of that block lies within $3/16<1/4$ of the
representative. Good rows of any other block are farther than
$\sqrt2-3/16>1/4$, so deletion preserves the induction.
After $K$ selections, each good row is closer to its own
representative than to all others. Only the $m$ bad rows can be
misclassified. Dividing by $n$ gives the stated misclassification rate.
\end{proof}

For the remaining reconstruction estimates, write
$v=A\one-\widehat A\one$ and $R_1=\widehat A+v\one^\top/n$.
These are proof quantities computed from the same fixed reconstruction;
they are not additional algorithmic updates.

\begin{lemma}[Reconstructed block totals]
\label{restore:fit-block-totals}
The hollow, degree-preserving reconstruction in
\eqref{restore:reconstruction} satisfies
\begin{equation}\label{restore:fit-global-absolute}
 \sum_{i,j}|\widehat A^{\mathrm{c}}_{ij}|\le Cn^2p.
\end{equation}
For every estimated partition, under the common label matching,
\begin{equation}\label{restore:fit-total-accuracy}
 \left|[\widehat Z^\top\widehat A^{\mathrm{c}}\widehat Z]_{k\ell}
                  -n_kn_\ell P_{k\ell}\right|
 \le Cn\{1+np\Mis(\widehat z,z)\}.
\end{equation}
The omitted diagonal is included in the bound, and the constant is
uniform over the estimated partition.
\end{lemma}

\begin{proof}
With $r_i=(R_1)_{ii}$, the reconstruction formula is equivalently
\[
 \widehat A^{\mathrm{c}}=R_1+\frac{r\one^\top}{n-1}
                   -\frac n{n-1}\diag(r),
\]
because its off-diagonal entries are $\widehat A_{ij}+(v_i+\widehat A_{ii})/(n-1)$.
Let $Q_{\rm tail}=(I-\Pi)Q(I-\Pi)$. Spectral invariance gives
\[
 v=G^{1/2}Q_{\rm tail}(I-\Pi)G^{1/2}\one.
\]
The bounds $\norm{Q_{\rm tail}}_{\op}\le C/\sqrt{np}$,
$\norm{(I-\Pi)G^{1/2}\one}_2\le C\sqrt n$ and
$\sum_iG_i\le Cn^2p$ imply $\sum_i|v_i|\le Cn$.
Also
\[
 \sum_i|\widehat A_{ii}|\le\sum_iG_i\Pi_{ii}
 \le2\tau K+\sum_i\frac{(G_i-\tau)^2}{G_i}\le Cn.
\]
Consequently $\sum_i|(R_1)_{ii}|\le Cn$; hollowing has total
absolute cost at most $2\sum_i|(R_1)_{ii}|$.
Since $|\widehat A_{ij}|\le a_ia_j$, \cref{restore:fit-wrong-mass} bounds
its total absolute mass by $Cn^2p$. The rank-one degree correction
has absolute mass $\sum_i|v_i|$, proving
\eqref{restore:fit-global-absolute}.

Let $E$ be the incorrectly assigned set. Each nonzero row of
$\widehat Z-Z$ has squared norm two, so the block-space residual and
\cref{restore:fit-wrong-mass} give
\[
 \norm{(I-\Pi)G^{1/2}\widehat Z}_F^2
 \le 2\norm{(I-\Pi)G^{1/2}Z}_F^2
       +2\norm{G^{1/2}(\widehat Z-Z)}_F^2
 \le C(n+np|E|).
\]
The exact identity
\[
 \widehat Z^\top(\widehat A-A)\widehat Z
 =-\{(I-\Pi)G^{1/2}\widehat Z\}^\top Q_{\rm tail}
       \{(I-\Pi)G^{1/2}\widehat Z\}
\]
therefore has norm at most $C(n+np|E|)/\sqrt{np}$.
Moving the incorrectly assigned vertices changes any raw block total
by at most $C\sum_E D_i\le C(n+np|E|)$.
Together with \eqref{restore:fit-block-event}, this bounds the error
for $\widehat A$ by $C(n+np|E|)$.

The degree correction changes a block total by
$[\widehat Z^\top v]_k n_\ell(\widehat z)/n$.
Using spectral invariance again gives
\[
 \norm{\widehat Z^\top v}_2
 \le \frac{C\sqrt{n(n+np|E|)}}{\sqrt{np}}
 \le C(n+np|E|)
\]
for all sufficiently large $np$. The global $O(n)$ hollowing cost
has already been proved. Finally, $|E|=n\Mis(\widehat z,z)$ gives
\eqref{restore:fit-total-accuracy}.
\end{proof}

For a tested vertex $i$, the local regularity condition is
\begin{equation}\label{restore:fit-local-row}
 D_i\le L_Dnp,\qquad
 a_i=\sqrt{G_i}\norm{U_{i\cdot}}_2\le L_U\sqrt p,
\end{equation}
where $L_D,L_U$ are fixed. It is imposed only at the tested target.

\begin{lemma}[Local absolute mass]
\label{restore:fit-local-mass}
On \eqref{restore:fit-local-row},
\begin{equation}\label{restore:fit-local-mass-bounds}
 \sum_j|\widehat A^{\mathrm{c}}_{ij}|\le Cnp,\qquad
 \sum_{j\in E}|\widehat A^{\mathrm{c}}_{ij}|\le C(1+p|E|)
       \quad\text{for every }E\subseteq[n].
\end{equation}
For the incorrectly assigned set the second bound is
$C\{1+np\Mis(\widehat z,z)\}$.
\end{lemma}

\begin{proof}
The inequality $|\widehat A_{ij}|\le a_ia_j$, the local bound on $a_i$ and
\cref{restore:fit-wrong-mass} prove both claims for $\widehat A$.
They also give $|v_i|\le D_i+\sum_j|\widehat A_{ij}|\le Cnp$.
The rank-one correction adds at most $Cnp$ to the full row mass and
$Cp|E|$ to its mass over $E$. Finally,
$|(R_1)_{ii}|\le a_i^2+|v_i|/n\le Cp$; hollowing changes either
bound by at most $2|(R_1)_{ii}|$. This proves the assertion.
\end{proof}

\begin{proposition}[Parameter and score errors in terms of misclassification]
\label{restore:fit-ordinary}
\label{restore:fit-score}
There is a fixed $c>0$ such that, uniformly over current labels with
$\Mis(\widehat z,z)\le c$, the symmetric Bernoulli block fit is unique
and interior for all sufficiently large $n$, and
\begin{equation}\label{restore:fit-ordinary-parameter}
 \max_{k,\ell}|\widehat P_{k\ell}-P_{k\ell}|
 \le C\left\{\frac1n+p\Mis(\widehat z,z)\right\}.
\end{equation}
At every target satisfying \eqref{restore:fit-local-row}, the fitted
log-likelihood contrast obeys
\begin{equation}\label{restore:fit-ordinary-score}
 \begin{aligned}
 &\sum_{j\ne i}\left\{
 \widehat A^{\mathrm{c}}_{ij}\log\frac{\widehat P_{k\widehat z_j}}
                              {\widehat P_{\ell\widehat z_j}}
 +(1-\widehat A^{\mathrm{c}}_{ij})\log\frac{1-\widehat P_{k\widehat z_j}}
                                   {1-\widehat P_{\ell\widehat z_j}}
 \right\}\\
 &\qquad=[\widehat A^{\mathrm{c}} w_{k\ell}]_i
       +\sum_{j\ne i}\log\frac{1-P_{kz_j}}{1-P_{\ell z_j}}
       +O\{1+np\Mis(\widehat z,z)\}.
 \end{aligned}
\end{equation}
Here $w_{k\ell}$ is the true log-odds vector in
\eqref{eq:sbm-weights}. The constants are uniform over the current labels on the same aggregate
event. This is a simultaneous deterministic assertion, so it also
holds when the labels or their error set are arbitrary functions of
$A$ and any seeding randomness. In particular, a misclassification
rate of order $1/(np)$ gives parameter error $O(1/n)$ and a bounded
score remainder. It does not assert conditional independence of
$\widehat P$ or of the initial labels from the tested row.
\end{proposition}

\begin{proof}
For one unordered fitted block pair, write $E_{k\ell}$ for its ordered
reconstructed total and $N_{k\ell}$ for the number of ordered
 off-diagonal pairs. There are two directions when $k\ne\ell$ and
one ordered within-group total when $k=\ell$. Consequently
\cref{restore:fit-block-totals} and the exact pair counts give
\[
 \begin{aligned}
 E_{k\ell}
 &=(2-\ind{k=\ell})n_kn_\ell P_{k\ell}
        +O\bigl(n\{1+np\Mis(\widehat z,z)\}\bigr),\\
 N_{k\ell}
 &=(2-\ind{k=\ell})n_k(\widehat z)n_\ell(\widehat z)+O(n).
 \end{aligned}
\]
For the common label matching,
$|n_k(\widehat z)-n_k|\le n\Mis(\widehat z,z)$.
Choose the fixed $c$ small enough that every fitted group has size
at least $n/(2\beta K)$. Thus $N_{k\ell}\asymp n^2$, and
\[
 \left|\frac{E_{k\ell}}{N_{k\ell}}-P_{k\ell}\right|
 \le C\left\{\frac1n+p\Mis(\widehat z,z)\right\}.
\]
After decreasing $c$ if necessary, this error is a sufficiently small
fixed multiple of $p$ for all large $np$. The ratio is then positive,
comparable to $p$, and strictly inside the algorithm's clipping
interval. This uses the stipulated sufficiently small $c_0$ and
$p\to0$, a consequence of $np^2\le C_1$.
The criterion $E_{k\ell}\log q+(N_{k\ell}-E_{k\ell})\log(1-q)$
has the unique interior maximizer $q=E_{k\ell}/N_{k\ell}$,
proving \eqref{restore:fit-ordinary-parameter}.

On the positive comparable-probability class, the resulting log-odds
coefficient error is at most
$C\{(np)^{-1}+\Mis(\widehat z,z)\}$.
Multiplication by the absolute row mass $Cnp$ costs at most
$C\{1+np\Mis(\widehat z,z)\}$.
Changing the incorrectly assigned neighbor labels alters bounded
coefficients on $E$, at cost at most $C(1+p|E|)$ by
\cref{restore:fit-local-mass}.
For the nonedge intercept, bounded derivatives of $\log(1-q)$ give
parameter cost at most $C\{1+np\Mis(\widehat z,z)\}$;
wrong neighbor labels add at most $Cp|E|$.
As $|E|=n\Mis(\widehat z,z)$, these bounds prove
\eqref{restore:fit-ordinary-score}. They hold simultaneously over
all stated input labels; no independence of the fitted quantities
from the tested row is used.
\end{proof}

\begin{lemma}[Uniform misclassification control for finite hard refits]
\label{restore:fit-invariant-basin}
\label{restore:fit-invariant}
There are fixed constants $c>0$ and $C<\infty$ such that, on the
aggregate event and for all sufficiently large $n$, every input
satisfying $\Mis(\widehat z,z)\le c$ yields, after one block fit and
Bernoulli assignment,
\[
 \Mis\!\left(\mathcal L(\widehat A^{\mathrm{c}},\widehat P,\widehat z),z\right)
       \le\frac{C}{np}.
\]
For each fixed $C_0$, start from any labels satisfying
$\Mis(\widehat z^{(0)},z)\le C_0/(np)$. All iterates then satisfy a
single bound $C/(np)$ under the same initial label matching, where
$C$ may also depend on $C_0$. This statement is simultaneous over
all such starting partitions, including same-graph randomized ones.
For every such sequence, at every target satisfying
\eqref{restore:fit-local-row}, an error at any updated iterate implies
\begin{equation}\label{restore:fit-canonical-error}
 [\widehat A^{\mathrm{c}} w_{z_i\ell}]_i+
 \sum_{j\ne i}\log\frac{1-P_{z_i z_j}}{1-P_{\ell z_j}}\le C
 \quad\text{for some }\ell\ne z_i.
\end{equation}
The threshold is the same for all such iterates and for any finite,
possibly data-dependent stopping rule after at least one update.
\end{lemma}

\begin{proof}
Choose large fixed $L_D,R$, with $L_D$ larger than twice the uniform
upper bounds on $\sum_jP_{z_i z_j}/(np)$ and $\tau/(np)$.
The degree energy bounds $|\{i:D_i>L_Dnp\}|\le C/p$.
Also $\norm{V_{i\cdot}}_2\le C/\sqrt n$, so
\eqref{restore:fit-basis} gives
$|\{i:\sqrt n\norm{U_{i\cdot}}_2>R\}|\le C/p$.
Every other target satisfies \eqref{restore:fit-local-row} for a
fixed $L_U$, because $G_i\le Cnp$ there. These counting exceptions
are independent of the input labels.

Best rank-$K$ approximation gives
$\norm{Q_K-G^{-1/2}ZPZ^\top G^{-1/2}}_{\op}\le C/\sqrt{np}$,
whence
\[
 \norm{G^{-1/2}(\widehat A-ZPZ^\top)Z}_F^2
 \le \frac{C}{np}\norm{G^{1/2}Z}_F^2\le Cn.
\]
The earlier expression for $v$ implies
$\norm{G^{-1/2}v}_2^2\le C/p$.
Since $\norm{\one^\top Z/n}_2\le1$, the degree correction
preserves the bound:
\begin{equation}\label{restore:fit-basin-action}
 \norm{G^{-1/2}(R_1-ZPZ^\top)Z}_F^2\le Cn.
\end{equation}
Hollowing changes the action of a bounded block-constant vector
at row $i$ by at most $C|(R_1)_{ii}|$.
Because $\sum_i|(R_1)_{ii}|\le Cn$, it changes that action by more
than a prescribed small multiple of $np$ on at most $C/p$ rows.
This auxiliary set is used only for counting misclassifications;
it is not charged as a probability exception in the sharp risk proof.

For $k=z_i$, the true population contrast is
\[
 [ZPZ^\top w_{k\ell}]_i+
       \sum_{j\ne i}\log\frac{1-P_{kz_j}}{1-P_{\ell z_j}}.
\]
It is a sum of incident Bernoulli Kullback--Leibler divergences plus
an $O(p)$ completed-diagonal term. Row separation and community-size
bounds give a lower bound $c_{\rm gap}np$.
At a regular target, \cref{restore:fit-ordinary} bounds the fitted
remainder by $C\{1+np\Mis(\widehat z,z)\}$, with $C$ independent
of the input labels. First choose the fixed $c$ in the hypothesis
sufficiently small compared with this population margin, then take
$np$ sufficiently large. The remainder is at most $c_{\rm gap}np/4$.
Outside the auxiliary hollowing set, a new error therefore requires
\[
 \norm{[(R_1-ZPZ^\top)Z]_{i\cdot}}_2\ge c'np
\]
for a fixed $c'>0$. Summing the squared requirements divided by
$G_i\le Cnp$ and using \eqref{restore:fit-basin-action} bounds the
number of such errors by $C/p$. Including the counting exceptions
and dividing by $n$ proves the asserted output rate. Its constant
does not depend on the input misclassification rate within the
specified small interval.

Choose a single fixed upper constant larger than $C_0$ and the
one-update output constant. For all sufficiently large $n$, its ratio
to $np$ is smaller than the fixed $c$ chosen above. The first input
satisfies this bound by hypothesis, and every later input satisfies
it by the one-update conclusion. Induction controls all iterates
without using any property of the initialization rule. The estimates
are deterministic on the aggregate event and use the same true
community numbering. A matching of the initial labels is chosen once
and propagated to subsequent labels; even in a tie, the selected
label remains a maximizer after this relabeling. No probabilistic
independence is inferred from that matching. For sufficiently large
$n$, the aligned error count is below $n/(2\beta K)$. This also makes
the matching uniquely minimizing: a different permutation loses at
least $n/(\beta K)-|E|$ currently correct labels and can gain at most
$|E|$ labels. Thus the same aligned count equals $n\Mis$ at each
iterate, rather than silently switching to a new optimizing permutation.

Along this sequence $np\Mis(\widehat z,z)$ is uniformly bounded.
Equation~\eqref{restore:fit-ordinary-score} thus has a single bounded
remainder for all input iterates. A wrong assignment implies
\eqref{restore:fit-canonical-error}. The same event covers finite
adaptive stopping, with no probability union over iterations.
\end{proof}

\section{Exponential accuracy of the aggregate event}
\label{sec:high-range}

\begin{lemma}[Exponentially accurate aggregate geometry above a logarithmic scale]
\label{restore:high-aggregate}
Fix any $A_0>0$ and $s>0$. Under the fixed positive compact structural
conditions, if ${np}\ge A_0\log n$ and $np^2\le C_1$, there is an aggregate
event $\mathcal E_s$ with
\[
 \PP(\mathcal E_s^c)\le C_s e^{-s{np}}
\]
on which the capped-adjacency geometry, weighted degree-energy bounds,
two-sided total-degree bounds, and true block-total bounds used in
\cref{restore:fit-geometry-one,restore:fit-geometry-two,restore:fit-block-event}
hold for the full graph and every graph obtained by deleting at most three
vertices. The constants may depend on $A_0,s,C_1$ and the structural
bounds. No condition is imposed on the smallest raw degree or the
largest spectral row.
\end{lemma}

\begin{proof}
The centered bounded-entry norm inequality of \citet[Corollary 3.12 and Remark 3.13]{bvh} gives
\[
 \PP\{\norm{A-{ZPZ^\top}}_{\op}>L_s\sqrt {np}\}
 \le e^{-(s+2){np}}
\]
after choosing a sufficiently large fixed $L_s$. Indeed, the row
variance sum is $O({np})$, the entry bound is one, and the prefactor $n$
in the norm-tail inequality is absorbed using $\log n\le {np}/A_0$.
The completed-mean diagonal has norm $O(p)$ and is absorbed in the
same threshold. No minimum raw degree is imposed.

The Bernoulli transform of the total number of edges gives, with
failure probability at most $Ce^{-c{n^2p}}$,
\[
 c{n^2p}\le\sum_iD_i\le C{n^2p}.
\]
For each of the fixed $K^2$ true block totals, Bernstein's inequality
for the independent upper-triangular edges, whose coefficients have
absolute value at most two, gives
\[
 \PP\left\{\left|[Z^\top(A-{ZPZ^\top})Z]_{k\ell}\right|>Ln\right\}
 \le 2\exp\left\{-\frac{cL^2n}{{np}+L}\right\}.
\]
The deterministic completed diagonal costs $O({np})$ and is included
by increasing $L$. Since ${np}\to\infty$ and $np^2\le C_1$,
the exponent is at least $c'L^2{np}/C_1$ for all sufficiently large $n$.
Increasing the fixed $L=L_s$ therefore makes the failure of all block
totals at most $C_se^{-(s+2){np}}$. The total-degree failure is also
bounded by $Ce^{-(s+2){np}}$ for large $n$.

These three full-graph events imply the required events on every
fixed-depth minor without a degree-wise probability union. To see this,
let $J=[n]\setminus T$ with $|T|\le3$, set $A_J=A_{JJ}$ and
$Z_J=Z_{J\cdot}$, and recompute all degrees and the cap
threshold on $J$. The centered matrix $A_J-{Z_JPZ_J^\top}$ is a principal
submatrix of $A-{ZPZ^\top}$, so its norm is at most $L_s\sqrt {np}$.
Removing $T$ changes the total raw degree by at most $2|T|n$;
as ${np}\to\infty$, the two-sided bound by fixed multiples of ${n^2p}$
therefore persists uniformly over $T$. Removing $T$ changes a true
block total of $A-{ZPZ^\top}$ by $O(|T|n)$ deterministically, since
the entries have bounded absolute value. Its $O_s(n)$ bound persists
as well. The surviving group masses differ from their original
values by $O(1)$ and remain comparable to $n$.

The recomputed average degree is consequently comparable to ${np}$.
For the fixed sufficiently large cap multiplier $b_0$, its threshold
$\tau_J=b_0\overline D_J$ lies between positive fixed multiples of
${np}$ and above every completed row sum $\sum_{j\prime\in J}P_{z_j z_{j\prime}}$.
Thus $G_{J,j}=\max(D_{J,j},\tau_J)\ge c{np}$ and
$\sum_jG_{J,j}\le C{n^2p}$. The raw norm bound immediately gives
\begin{align*}
 \norm{G_J^{-1/2}(A_J-{Z_JPZ_J^\top})G_J^{-1/2}}_{\op}
   &\le C_s{(np)}^{-1/2},\\
 \sum_{j\in J}\frac{(D_{J,j}-\sum_{j\prime\in J}P_{z_j z_{j\prime}})^2}{G_{J,j}}
   &=\norm{G_J^{-1/2}(A_J-{Z_JPZ_J^\top})\one}_2^2
     \le C_sn.
\end{align*}
The cap exceeds the completed population degrees, so coordinatewise
\[
 \frac{(G_{J,j}-\tau_J)^2}{G_{J,j}}
 \le\frac{(D_{J,j}-\sum_{j\prime\in J}P_{z_j z_{j\prime}})^2}{G_{J,j}}.
\]
Indeed, the left side is zero when $D_{J,j}\le\tau_J$; otherwise
$0\le D_{J,j}-\tau_J\le D_{J,j}-\sum_{j\prime\in J}P_{z_j z_{j\prime}}$. This proves the
other weighted energy bound. The order-${np}$ signal gap of the
capped adjacency matrix (a fixed gap after dividing by $\tau_J$)
and the residual
$\norm{(I-U_JU_J^\top)G_J^{1/2}Z_J}_F\le C_s\sqrt n$
now follow by the deterministic Gram-matrix and projection arguments
in \cref{prop:sparse-floor}. Those arguments use only
$G_{J,j}\ge\tau_J\asymp {np}$, total degree, the displayed rescaled
noise bound, and the energy relative to $\tau_J$, all just proved.
Intersecting the stated aggregate consequences defines
$\mathcal E_s$. The three full-graph events imply it simultaneously
for every indicated minor, proving its claimed failure bound.

For a deleted graph used as independent training data, form the
analogous event using that graph and its further minors. The same
argument gives failure probability $C_se^{-s{np}}$, and this event is
measurable from the training graph alone. The full-graph common event
is used only for deterministic implications, never as a condition
that would preserve independence of an incident row.
\end{proof}

\section{Completion of the oracle-risk bound}
\label{sec:probability}

The final step uses a deleted graph only as a proof device. The
implemented procedure retains one full-graph eigensystem.
For each target $i$, put $B=A_{-i,-i}$ and recompute its degrees,
cap threshold, weights and projection $H_B$. Throughout this section, all constants
are uniform over the fixed structural class.

\begin{lemma}[Ordinary Bernoulli tilted cancellation]
\label{lem:sbm-floor-cancellation}
Let $w$ be a true block-constant Bernoulli log-odds vector on the
deleted graph and put
\[
 f=w-\bar w\one,\qquad \varepsilon=(H_B-I)f,
 \qquad \bar w=(n-1)^{-1}\one^\top w.
\]
On the deleted aggregate event, under the exact Chernoff tilt of
the raw-row oracle contrast, its Bernoulli probabilities $r_j$
satisfy
\begin{equation}\label{eq:sbm-floor-tilt}
 \left|\sum_{j\ne i}r_j\varepsilon_j\right|\le C,
 \qquad \sum_{j\ne i}r_j\varepsilon_j^2\le C.
\end{equation}
\end{lemma}

\begin{proof}
The tilted probabilities $r_j$ are constant within each true block
and satisfy $r_j\le Cp$. Write $\tau$ and $G$ for the threshold
and diagonal matrix recomputed from $B$. Since $r$ is block constant with $\norm r_\infty\le Cp$,
weighted orthogonality gives directly
\[
 r^\top\varepsilon
 =-\frac1\tau\big\langle(I-H_B)r,(I-H_B)f\big\rangle_G
       +\frac1\tau r^\top(\tau I-G)\varepsilon.
\]
The residual bounds give
$\norm{(I-H_B)r}_G\le Cp\sqrt n$ and
$\norm{(I-H_B)f}_G\le C\sqrt n$.
Also $\sum_j(G_j-\tau)^2/G_j\le Cn$, so
\[
 \norm{G^{-1/2}(\tau I-G)r}_2\le Cp\sqrt n,
 \qquad \norm\varepsilon_G\le C\sqrt n.
\]
Dividing either product by $\tau\asymp np$ gives $O(1)$.
This proves the first bound. Since $r_j\le Cp$ and
$\norm\varepsilon_2^2\le \frac{C}{p}$, the second bound follows as well.
The deterministic nonedge intercept remains in the oracle score
and is not absorbed into this cancellation.
When the cap is inactive on the deleted graph, $G=\tau I$ and
$H_B=\Pi_B=U_BU_B^\top$. The second term vanishes, leaving the
ordinary Euclidean adjacency identity
\[
 r^\top\varepsilon
 =-\big\langle(I-\Pi_B)r,(I-\Pi_B)f\big\rangle.
\]
\end{proof}

\begin{proof}[Proof of \cref{restore:sbm-theorem}]
Fix $C_0$ from the statement and write
\[
 \mathcal H_0=\{\Mis(\widehat z^{(0)},z)\le C_0/(np)\}.
\]
This event depends on the full graph and any initialization randomness.
We never condition on $\mathcal H_0$. Since $0\le\Mis\le1$, its
complement will contribute at most $\PP(\mathcal H_0^c)$ once.
Fix $A_0>0$. First consider ${np}\le A_0\log n$, with ${np}\to\infty$.
Let $\mathcal E$ collect the aggregate conditions
\eqref{restore:fit-geometry-one}--\eqref{restore:fit-block-event},
two-sided total degree, and their recomputed versions after deletion
of at most three vertices. For each prescribed fixed $r>0$,
\cref{prop:sparse-floor} gives these geometric conditions with
failure probability at most $C_rn^{-r}$, by increasing its
probability parameter to cover the polynomial number of minors.
For any true block total, independent-edge Bernstein gives
\[
 \PP\{|[Z^\top(A-{ZPZ^\top})Z]_{k\ell}|>Ln\}
 \le2\exp\{-cL^2n/({np}+L)\}.
\]
In the present range this is smaller than any fixed inverse power
of $n$, as is the total-degree failure. The same holds for the
required minors. For a target $i$, define $\mathcal E_i$ using
only its deleted graph and its further deletions. Use the same fixed
bounds in these definitions, so that $\mathcal E\subseteq\mathcal E_i$
for every $i$, and $\mathcal E_i$ is measurable from $B=A_{-i,-i}$.
Then
\begin{equation}\label{eq:sbm-geometry-payment}
 \PP(\mathcal E^c)+\frac1n\sum_i\PP(\mathcal E_i^c)
 \le C_rn^{-r}.
\end{equation}

On $\mathcal H_0\cap\mathcal E$, \cref{restore:fit-invariant}
gives the initial partition and every subsequent iterate a common
bound $\Mis(\widehat z,z)\le C/(np)$. The constant depends on
$C_0$, not on the rule producing the labels. At target
$i$ impose only $D_i\le L_D{np}$, its local spectral-row condition
\eqref{restore:fit-local-row}, and the requirement that the
incident row hits no exceptional deleted spectral row. By
\cref{restore:tail-energy-lemma,restore:tail-hit-corollary,restore:tail-target-corollary},
the average probability cost of these spectral exceptions is at
most $C_se^{-s{np}}$ for any prescribed fixed $s$, on choosing fixed
thresholds sufficiently large. Indeed, $G_j\ge\tau\asymp {np}$ gives
\eqref{restore:loo-hit} outside a deleted exceptional set
$\{j:\sqrt n\norm{(U_B)_{j\cdot}}_2>R\}$, whose expected hit
mass is controlled by \eqref{restore:tail-coefficient-count}.
The degree tail has the same bound by the single-row Bernoulli
Chernoff inequality. Polynomial factors in ${np}$ are absorbed by
first choosing a larger exponential rate. These are averages of
target probabilities, without a union over all tested vertices.

Let $\mathcal R_i$ denote the stated local target and hit-row conditions.
On $\mathcal H_0\cap\mathcal E\cap\mathcal R_i$, the deterministic
fitted-score bound and \cref{restore:loo-corrected} imply that an
aligned error at $i$ after any positive number of updates requires,
for some $\ell\ne k=z_i$,
\begin{equation}\label{eq:sbm-independent-score}
 T_{k\ell}+\sum_{j\ne i}A_{ij}\varepsilon_j\le C,
 \qquad \varepsilon=(H_B-I)(w_{k\ell,-i}-\bar w_{k\ell,-i}\one).
\end{equation}
Here $T_{k\ell}$ is realized from this target's actual row, with its
full nonedge term and $n_k-1$ own-community trials. Its law is the
community-$k$ law of Section~\ref{sec:statements}; no independent
copy is substituted. For each fixed competitor, the displayed
coefficients $\varepsilon=\varepsilon^{k\ell}(B)$ depend only on
$B$ and the fixed true parameters and labels, not on any fitted
partition. The uniform error threshold $C$ depends on $C_0$ and
fixed local bounds, never on the initial partition or the stopping time.

The precise conditioning step is
\begin{equation}\label{eq:same-graph-conditioning}
\begin{aligned}
 &\PP\{\text{aligned error at }i,\mathcal H_0,\mathcal E,\mathcal R_i\}\\
 &\quad\le\sum_{\ell\ne k}\E\!\left[\ind{\mathcal E_i}
  \PP_k\!\left\{T_{k\ell}+\sum_{j\ne i}A_{ij}\varepsilon_j^{k\ell}(B)
                 \le C\,\middle|\,B\right\}\right].
\end{aligned}
\end{equation}
First use the deterministic inclusion, valid for every successful
initial partition and every refit. Then discard $\mathcal H_0$,
$\mathcal E$ and $\mathcal R_i$ from the error event, retaining only
$\mathcal E_i$ because $\mathcal E\subseteq\mathcal E_i$.
The resulting event has no initialization or stopping-time dependence.
Only now condition on $B$. For fixed $z,P$, the incident edges are
independent Bernoulli variables independent of $B$, with unchanged
probabilities $P_{kz_j}$. This independence is not claimed after
conditioning on the initial labels, their matching, initialization
success, full degrees, or any full-graph spectral event.
The right side of \eqref{eq:same-graph-conditioning} is therefore
valid after integrating all initialization randomness, without any
leave-one-out stability assumption on the initializer.

The tilted mean and quadratic bounds in
\cref{lem:sbm-floor-cancellation} give the moment hypotheses of
\cref{lem:sparse-truncated-transfer}. Its exceptional-coefficient
hit term is at most $C_se^{-s{np}}$ by
\cref{restore:tail-hit-corollary}, after selecting a fixed
coefficient threshold. That truncation is used only in the proof.
Apply the sharp transfer to \eqref{eq:sbm-independent-score} using
the exact self-excluded row law. By the one-coordinate comparison in
Appendix~\ref{sec:sharp}, its exponential risk factor can be replaced
by $e^{-I_{k\ell}}$ at a uniformly bounded multiplicative cost.
Averaging targets and taking the fixed-$K$ union over competitors,
we obtain
\begin{equation}\label{eq:sbm-final-payment}
 \begin{split}
 \E\Mis(\widehat z,z)
 \lesssim{}&\sum_{k=1}^K\frac{n_k}{n}
                     \sum_{\ell\ne k}e^{-I_{k\ell}}\\
 &+e^{-snp}+\PP(\mathcal H_0^c)
       +\PP(\mathcal E^c)+\frac1n\sum_i\PP(\mathcal E_i^c).
 \end{split}
\end{equation}
Here the implicit factor may also depend on the chosen fixed
$s$ and the associated fixed local thresholds, never on $n$ or the
number of refits. We grouped the local testing terms by $k=z_i$;
the graph-dependent exceptional events retain their vertex indices.
The comparison in \eqref{eq:oracle-pairs} bounds the testing term by
a constant multiple of $R_{\mathrm{oracle}}$, for the specified tie rule.
The initialization complement was charged using $\Mis\le1$;
its coefficient can be one before collecting the other fixed factors.
There is no factor $n$ or number of iterations multiplying it.

On the compact structural class, choose a fixed $C$ such that
\[
 I_{k\ell}-\frac12\log(np)\le Cnp
\]
uniformly over the tested pairs. Consequently
\[
 R_{\mathrm{oracle}}\gtrsim\exp\{-Cnp-\tfrac12\log(np)\}.
\]
Choose $s>C+1$ and $r>A_0(C+1)+2$. When
$np\le A_0\log n$,
\[
 \sqrt{np}\left(e^{-(s-C)np}+n^{-r}e^{Cnp}\right)\longrightarrow0,
\]
which pays the spectral and geometry remainders in
\eqref{eq:sbm-final-payment}. The initialization failure term remains
explicit, as required in \eqref{eq:target}.

For ${np}>A_0\log n$, use \cref{restore:high-aggregate} in place of
the polynomially accurate geometry. It gives
$\PP(\mathcal E^c)+n^{-1}\sum_i\PP(\mathcal E_i^c)\le C_{s'}e^{-s'{np}}$ for any fixed
$s'>0$, with training events still measurable from the respective
deleted graphs. The row-tail proof uses deterministic aggregate
geometry and applies innovation concentration before restricting
to that event. Its estimates, the local LOO comparison and the
fitting argument therefore remain valid for this stronger event.
Equation \eqref{eq:sbm-final-payment} follows with the new geometry
remainder. Choosing $s,s'>C+1$ pays both exponential terms.
The two proof ranges cover every sequence ${np}\to\infty$ under
$np^2\le C_1$, proving the theorem without a logarithmic degree
lower bound, with the separate initialization failure term retained.
\end{proof}

\begin{proof}[Proof of \cref{cor:kmpp-oracle}]
Choose the polynomial or exponential aggregate event used above, with
fixed accuracy parameters large enough that its complement has
probability $O(R_{\mathrm{oracle}})$. Its bounds also give
$\bar D\ge cnp$ for a fixed $c>0$. Apply
\cref{lem:kmpp-initializer} with the corresponding fixed $C_0$.
The restart count is a measurable function of $A$, so conditional on
$A$ the runs remain independent and the lemma yields
\[
 \PP\!\left\{\Mis(\widehat z^{(0)},z)>\frac{C_0}{np}\right\}
 \le \PP(\mathcal E^c)+\E\!\left[\ind{\mathcal E}2^{-m(A)}\right].
\]
For $m(A)=\max\{1,\lceil C\bar D\rceil\}$, the second term is at most
$\exp\{-cC(\log2)np\}$. The lower bound
$R_{\mathrm{oracle}}\gtrsim (np)^{-1/2}\exp\{-C'np\}$ in the
preceding proof is uniform over the structural class. Choosing the
fixed restart multiplier large enough that $cC\log2>C'+1$ makes
the seeding contribution $O(R_{\mathrm{oracle}})$. Substitution into
\cref{restore:sbm-theorem} proves the corollary. This argument uses
only $np\to\infty$, not a logarithmic density floor. It controls the
additional seeding randomness; it does not replace the graph-event
estimates or the uniform fitted-score and leave-one-out arguments.
\end{proof}

\bibliographystyle{plainnat}
\bibliography{reference}

\begin{thebibliography}{40}
\providecommand{\natexlab}[1]{#1}
\providecommand{\url}[1]{\texttt{#1}}
\expandafter\ifx\csname urlstyle\endcsname\relax
  \providecommand{\doi}[1]{doi: #1}\else
  \providecommand{\doi}{doi: \begingroup \urlstyle{rm}\Url}\fi

\bibitem[Abbe(2018)]{abbe2018}
Emmanuel Abbe.
\newblock Community detection and stochastic block models: Recent developments.
\newblock \emph{Journal of Machine Learning Research}, 18\penalty0
  (177):\penalty0 1--86, 2018.
\newblock URL \url{https://jmlr.org/papers/v18/16-480.html}.

\bibitem[Abbe and Sandon(2015)]{abbesandon2015}
Emmanuel Abbe and Colin Sandon.
\newblock Community detection in general stochastic block models: Fundamental
  limits and efficient algorithms for recovery.
\newblock In \emph{Proceedings of the 56th Annual IEEE Symposium on Foundations
  of Computer Science}, pages 670--688, 2015.
\newblock URL \url{https://arxiv.org/abs/1503.00609}.

\bibitem[Abbe et~al.(2016)Abbe, Bandeira, and Hall]{abh2016}
Emmanuel Abbe, Afonso~S. Bandeira, and Georgina Hall.
\newblock Exact recovery in the stochastic block model.
\newblock \emph{IEEE Transactions on Information Theory}, 62\penalty0
  (1):\penalty0 471--487, 2016.
\newblock URL \url{https://arxiv.org/abs/1405.3267}.

\bibitem[Abbe et~al.(2020)Abbe, Fan, Wang, and Zhong]{afwz}
Emmanuel Abbe, Jianqing Fan, Kaizheng Wang, and Yiqiao Zhong.
\newblock Entrywise eigenvector analysis of random matrices with low expected
  rank.
\newblock \emph{The Annals of Statistics}, 48\penalty0 (3):\penalty0
  1452--1474, 2020.
\newblock \doi{10.1214/19-AOS1854}.
\newblock URL \url{https://arxiv.org/abs/1709.09565}.

\bibitem[Amini and Levina(2018)]{aminilevina2018}
Arash~A. Amini and Elizaveta Levina.
\newblock On semidefinite relaxations for the block model.
\newblock \emph{The Annals of Statistics}, 46\penalty0 (1):\penalty0 149--179,
  2018.
\newblock URL \url{https://arxiv.org/abs/1406.5647}.

\bibitem[Amini et~al.(2013)Amini, Chen, Bickel, and Levina]{amini2013}
Arash~A. Amini, Aiyou Chen, Peter~J. Bickel, and Elizaveta Levina.
\newblock Pseudo-likelihood methods for community detection in large sparse
  networks.
\newblock \emph{The Annals of Statistics}, 41\penalty0 (4):\penalty0
  2097--2122, 2013.
\newblock URL \url{https://arxiv.org/abs/1207.2340}.

\bibitem[Arthur and Vassilvitskii(2007)]{arthur2007}
David Arthur and Sergei Vassilvitskii.
\newblock {k-means++}: The advantages of careful seeding.
\newblock In \emph{Proceedings of the Eighteenth Annual ACM-SIAM Symposium on
  Discrete Algorithms}, pages 1027--1035. Society for Industrial and Applied
  Mathematics, 2007.

\bibitem[Bandeira and van Handel(2016)]{bvh}
Afonso~S. Bandeira and Ramon van Handel.
\newblock Sharp nonasymptotic bounds on the norm of random matrices with
  independent entries.
\newblock \emph{The Annals of Probability}, 44\penalty0 (4):\penalty0
  2479--2506, 2016.
\newblock \doi{10.1214/15-AOP1025}.
\newblock URL \url{https://arxiv.org/abs/1408.6185}.
\newblock Corollary 3.12 and Remark 3.13; arXiv:1408.6185v4.

\bibitem[Bickel and Chen(2009)]{bickelchen2009}
Peter~J. Bickel and Aiyou Chen.
\newblock A nonparametric view of network models and {Newman--Girvan} and other
  modularities.
\newblock \emph{Proceedings of the National Academy of Sciences}, 106\penalty0
  (50):\penalty0 21068--21073, 2009.
\newblock \doi{10.1073/pnas.0907096106}.
\newblock URL \url{https://pmc.ncbi.nlm.nih.gov/articles/PMC2795514/}.

\bibitem[Bickel et~al.(2013)Bickel, Choi, Chang, and Zhang]{bickel2013}
Peter~J. Bickel, David Choi, Xiangyu Chang, and Hai Zhang.
\newblock Asymptotic normality of maximum likelihood and its variational
  approximation for stochastic blockmodels.
\newblock \emph{The Annals of Statistics}, 41\penalty0 (4):\penalty0
  1922--1943, 2013.
\newblock \doi{10.1214/13-AOS1124}.
\newblock URL \url{https://arxiv.org/abs/1207.0865}.

\bibitem[Chin et~al.(2015)Chin, Rao, and Vu]{chin2015}
Peter Chin, Anup Rao, and Van Vu.
\newblock Stochastic block model and community detection in sparse graphs: A
  spectral algorithm with optimal rate of recovery.
\newblock In \emph{Proceedings of the 28th Conference on Learning Theory},
  volume~40 of \emph{Proceedings of Machine Learning Research}, pages 391--423,
  2015.
\newblock URL \url{https://proceedings.mlr.press/v40/Chin15.html}.

\bibitem[Decelle et~al.(2011)Decelle, Krzakala, Moore, and
  Zdeborov\'{a}]{decelle2011}
Aur\'{e}lien Decelle, Florent Krzakala, Cristopher Moore, and Lenka
  Zdeborov\'{a}.
\newblock Asymptotic analysis of the stochastic block model for modular
  networks and its algorithmic applications.
\newblock \emph{Physical Review E}, 84:\penalty0 066106, 2011.
\newblock \doi{10.1103/PhysRevE.84.066106}.
\newblock URL \url{https://arxiv.org/abs/1109.3041}.

\bibitem[Gao et~al.(2017)Gao, Ma, Zhang, and Zhou]{gao2017}
Chao Gao, Zongming Ma, Anderson~Y. Zhang, and Harrison~H. Zhou.
\newblock Achieving optimal misclassification proportion in stochastic block
  models.
\newblock \emph{Journal of Machine Learning Research}, 18\penalty0
  (60):\penalty0 1--45, 2017.
\newblock URL \url{https://jmlr.org/papers/v18/16-245.html}.

\bibitem[Gao et~al.(2018)Gao, Ma, Zhang, and Zhou]{gmzz2018}
Chao Gao, Zongming Ma, Anderson~Y. Zhang, and Harrison~H. Zhou.
\newblock Community detection in degree-corrected block models.
\newblock \emph{The Annals of Statistics}, 46\penalty0 (5):\penalty0
  2153--2185, 2018.
\newblock \doi{10.1214/17-AOS1615}.
\newblock URL \url{https://arxiv.org/abs/1607.06993}.

\bibitem[Hajek et~al.(2016)Hajek, Wu, and Xu]{hajek2016}
Bruce Hajek, Yihong Wu, and Jiaming Xu.
\newblock Achieving exact cluster recovery threshold via semidefinite
  programming.
\newblock \emph{IEEE Transactions on Information Theory}, 62\penalty0
  (5):\penalty0 2788--2797, 2016.
\newblock \doi{10.1109/TIT.2016.2546280}.
\newblock URL
  \url{https://experts.illinois.edu/en/publications/achieving-exact-cluster-recovery-threshold-via-semidefinite-progr/}.

\bibitem[Holland et~al.(1983)Holland, Laskey, and Leinhardt]{holland1983}
Paul~W. Holland, Kathryn~Blackmond Laskey, and Samuel Leinhardt.
\newblock Stochastic blockmodels: First steps.
\newblock \emph{Social Networks}, 5\penalty0 (2):\penalty0 109--137, 1983.
\newblock \doi{10.1016/0378-8733(83)90021-7}.
\newblock URL
  \url{https://www.sciencedirect.com/science/article/pii/0378873383900217}.

\bibitem[Jin(2015)]{jin2015}
Jiashun Jin.
\newblock Fast community detection by {SCORE}.
\newblock \emph{The Annals of Statistics}, 43\penalty0 (1):\penalty0 57--89,
  2015.
\newblock \doi{10.1214/14-AOS1265}.
\newblock URL \url{https://arxiv.org/abs/1211.5803}.

\bibitem[Joseph and Yu(2016)]{joseph2016}
Antony Joseph and Bin Yu.
\newblock Impact of regularization on spectral clustering.
\newblock \emph{The Annals of Statistics}, 44\penalty0 (4):\penalty0
  1765--1791, 2016.
\newblock \doi{10.1214/16-AOS1447}.
\newblock URL \url{https://arxiv.org/abs/1312.1733}.

\bibitem[Karrer and Newman(2011)]{karrer2011}
Brian Karrer and M.~E.~J. Newman.
\newblock Stochastic blockmodels and community structure in networks.
\newblock \emph{Physical Review E}, 83:\penalty0 016107, 2011.
\newblock \doi{10.1103/PhysRevE.83.016107}.
\newblock URL
  \url{https://journals.aps.org/pre/abstract/10.1103/PhysRevE.83.016107}.

\bibitem[Krzakala et~al.(2013)Krzakala, Moore, Mossel, Neeman, Sly,
  Zdeborov\'{a}, and Zhang]{krzakala2013}
Florent Krzakala, Cristopher Moore, Elchanan Mossel, Joe Neeman, Allan Sly,
  Lenka Zdeborov\'{a}, and Pan Zhang.
\newblock Spectral redemption in clustering sparse networks.
\newblock \emph{Proceedings of the National Academy of Sciences}, 110\penalty0
  (52):\penalty0 20935--20940, 2013.
\newblock \doi{10.1073/pnas.1312486110}.
\newblock URL \url{https://pmc.ncbi.nlm.nih.gov/articles/PMC3876200/}.

\bibitem[Le et~al.(2017)Le, Levina, and Vershynin]{llv2017}
Can~M. Le, Elizaveta Levina, and Roman Vershynin.
\newblock Concentration and regularization of random graphs.
\newblock \emph{Random Structures \& Algorithms}, 51\penalty0 (3):\penalty0
  538--561, 2017.
\newblock \doi{10.1002/rsa.20713}.
\newblock URL \url{https://arxiv.org/abs/1506.00669}.
\newblock Theorem 2.1 and Remark 2.2; arXiv:1506.00669.

\bibitem[Lei and Rinaldo(2015)]{lei2015}
Jing Lei and Alessandro Rinaldo.
\newblock Consistency of spectral clustering in stochastic block models.
\newblock \emph{The Annals of Statistics}, 43\penalty0 (1):\penalty0 215--237,
  2015.
\newblock URL \url{https://arxiv.org/abs/1312.2050}.

\bibitem[Massouli\'{e}(2014)]{massoulie2014}
Laurent Massouli\'{e}.
\newblock Community detection thresholds and the weak {Ramanujan} property.
\newblock In \emph{Proceedings of the 46th Annual ACM Symposium on Theory of
  Computing}, pages 694--703, 2014.
\newblock URL \url{https://arxiv.org/abs/1311.3085}.

\bibitem[Mossel et~al.(2018)Mossel, Neeman, and Sly]{mns2018}
Elchanan Mossel, Joe Neeman, and Allan Sly.
\newblock A proof of the block model threshold conjecture.
\newblock \emph{Combinatorica}, 38\penalty0 (3):\penalty0 665--708, 2018.
\newblock URL \url{https://arxiv.org/abs/1311.4115}.

\bibitem[Nowicki and Snijders(2001)]{nowicki2001}
Krzysztof Nowicki and Tom A.~B. Snijders.
\newblock Estimation and prediction for stochastic blockstructures.
\newblock \emph{Journal of the American Statistical Association}, 96\penalty0
  (455):\penalty0 1077--1087, 2001.
\newblock \doi{10.1198/016214501753208735}.
\newblock URL
  \url{https://www.tandfonline.com/doi/abs/10.1198/016214501753208735}.

\bibitem[Qin and Rohe(2013)]{qin2013}
Tai Qin and Karl Rohe.
\newblock Regularized spectral clustering under the degree-corrected stochastic
  blockmodel.
\newblock In \emph{Advances in Neural Information Processing Systems},
  volume~26, pages 3120--3128, 2013.
\newblock URL
  \url{https://papers.nips.cc/paper_files/paper/2013/hash/0ed9422357395a0d4879191c66f4faa2-Abstract.html}.

\bibitem[Rohe et~al.(2011)Rohe, Chatterjee, and Yu]{rohe2011}
Karl Rohe, Sourav Chatterjee, and Bin Yu.
\newblock Spectral clustering and the high-dimensional stochastic blockmodel.
\newblock \emph{The Annals of Statistics}, 39\penalty0 (4):\penalty0
  1878--1915, 2011.
\newblock URL \url{https://arxiv.org/abs/1007.1684}.

\bibitem[Rubin-Delanchy et~al.(2022)Rubin-Delanchy, Cape, Tang, and
  Priebe]{rubindelanchy2022}
Patrick Rubin-Delanchy, Joshua Cape, Minh Tang, and Carey~E. Priebe.
\newblock A statistical interpretation of spectral embedding: The generalised
  random dot product graph.
\newblock \emph{Journal of the Royal Statistical Society: Series B (Statistical
  Methodology)}, 84\penalty0 (4):\penalty0 1446--1473, 2022.
\newblock \doi{10.1111/rssb.12509}.
\newblock URL \url{https://academic.oup.com/jrsssb/article/84/4/1446/7073272}.

\bibitem[Snijders and Nowicki(1997)]{snijders1997}
Tom A.~B. Snijders and Krzysztof Nowicki.
\newblock Estimation and prediction for stochastic blockmodels for graphs with
  latent block structure.
\newblock \emph{Journal of Classification}, 14:\penalty0 75--100, 1997.
\newblock \doi{10.1007/s003579900004}.
\newblock URL
  \url{https://www.lunduniversity.lu.se/publication/ea745206-a48e-4a61-a3d9-fdb43967a910}.

\bibitem[Sussman et~al.(2012)Sussman, Tang, Fishkind, and Priebe]{sussman2012}
Daniel~L. Sussman, Minh Tang, Donniell~E. Fishkind, and Carey~E. Priebe.
\newblock A consistent adjacency spectral embedding for stochastic blockmodel
  graphs.
\newblock \emph{Journal of the American Statistical Association}, 107\penalty0
  (499):\penalty0 1119--1128, 2012.
\newblock \doi{10.1080/01621459.2012.699795}.
\newblock URL \url{https://arxiv.org/abs/1108.2228}.

\bibitem[Sussman et~al.(2014)Sussman, Tang, and Priebe]{sussman2014}
Daniel~L. Sussman, Minh Tang, and Carey~E. Priebe.
\newblock Consistent latent position estimation and vertex classification for
  random dot product graphs.
\newblock \emph{IEEE Transactions on Pattern Analysis and Machine
  Intelligence}, 36\penalty0 (1):\penalty0 48--57, 2014.
\newblock \doi{10.1109/TPAMI.2013.135}.
\newblock URL \url{https://arxiv.org/abs/1207.6745}.

\bibitem[Suwan et~al.(2016)Suwan, Lee, Tang, Sussman, Tang, and
  Priebe]{suwan2016}
Shakira Suwan, Dominic~S. Lee, Runze Tang, Daniel~L. Sussman, Minh Tang, and
  Carey~E. Priebe.
\newblock Empirical {Bayes} estimation for the stochastic blockmodel.
\newblock \emph{Electronic Journal of Statistics}, 10\penalty0 (1):\penalty0
  761--782, 2016.
\newblock \doi{10.1214/16-EJS1115}.
\newblock URL \url{https://arxiv.org/abs/1405.6070}.

\bibitem[Young and Scheinerman(2007)]{young2007}
Stephen~J. Young and Edward~R. Scheinerman.
\newblock Random dot product graph models for social networks.
\newblock In Anthony Bonato and Fan R.~K. Chung, editors, \emph{Algorithms and
  Models for the Web-Graph}, volume 4863 of \emph{Lecture Notes in Computer
  Science}, pages 138--149. Springer, 2007.
\newblock \doi{10.1007/978-3-540-77004-6_11}.
\newblock URL
  \url{https://link.springer.com/chapter/10.1007/978-3-540-77004-6_11}.

\bibitem[Zhang and Zhou(2016)]{zhangzhou2016}
Anderson~Y. Zhang and Harrison~H. Zhou.
\newblock Minimax rates of community detection in stochastic block models.
\newblock \emph{The Annals of Statistics}, 44\penalty0 (4):\penalty0
  1819--1857, 2016.
\newblock URL \url{https://arxiv.org/abs/1507.05313}.

\bibitem[Zhang(2024)]{zhang}
Anderson~Ye Zhang.
\newblock Fundamental limits of spectral clustering in stochastic block models.
\newblock \emph{IEEE Transactions on Information Theory}, 70\penalty0
  (10):\penalty0 7320--7348, 2024.
\newblock \doi{10.1109/TIT.2024.3425581}.
\newblock URL \url{https://arxiv.org/abs/2301.09289}.

\bibitem[Zhao et~al.(2012)Zhao, Levina, and Zhu]{zhao2012}
Yunpeng Zhao, Elizaveta Levina, and Ji~Zhu.
\newblock Consistency of community detection in networks under degree-corrected
  stochastic block models.
\newblock \emph{The Annals of Statistics}, 40\penalty0 (4):\penalty0
  2266--2292, 2012.
\newblock URL \url{https://arxiv.org/abs/1110.3854}.

\bibitem[Zhou and Amini(2019)]{zhouamini2019}
Zhixin Zhou and Arash~A. Amini.
\newblock Analysis of spectral clustering algorithms for community detection:
  The general bipartite setting.
\newblock \emph{Journal of Machine Learning Research}, 20\penalty0
  (47):\penalty0 1--47, 2019.
\newblock URL \url{https://jmlr.org/papers/v20/18-170.html}.

\bibitem[Zhou and Amini(2020)]{zhouamini2020}
Zhixin Zhou and Arash~A. Amini.
\newblock Optimal bipartite network clustering.
\newblock \emph{Journal of Machine Learning Research}, 21\penalty0
  (40):\penalty0 1--68, 2020.
\newblock URL \url{https://jmlr.org/papers/v21/19-299.html}.

\bibitem[Zhou and Li(2020)]{zhouli}
Zhixin Zhou and Ping Li.
\newblock Rate optimal {Chernoff} bound and application to community detection
  in the stochastic block models.
\newblock \emph{Electronic Journal of Statistics}, 14\penalty0 (1):\penalty0
  1302--1347, 2020.
\newblock \doi{10.1214/20-EJS1686}.
\newblock URL \url{https://arxiv.org/abs/1812.11269}.

\bibitem[Zhou et~al.(2021)Zhou, Zhou, Li, and Zhang]{zhou2021subspace}
Zhixin Zhou, Fan Zhou, Ping Li, and Cun-Hui Zhang.
\newblock Rate-optimal subspace estimation on random graphs.
\newblock In \emph{Advances in Neural Information Processing Systems},
  volume~34, pages 20283--20294, 2021.
\newblock URL
  \url{https://proceedings.neurips.cc/paper_files/paper/2021/hash/a97f6e2fedcabc887911dc9b5fd3ccc3-Abstract.html}.

\end{thebibliography}
\end{document}